\documentclass{amsart}

\usepackage{amsmath, amsthm, amsaddr}
\usepackage{caption, subcaption, amsfonts}
\usepackage{epstopdf}
\usepackage{pgf, float}
\usepackage{hyperref}

\usepackage{tikz}
\usetikzlibrary{matrix,arrows}
\usetikzlibrary{decorations.pathmorphing,mindmap}
\usetikzlibrary{mindmap,matrix,arrows,decorations.pathmorphing,patterns,fadings}
\usepackage[normalem]{ulem} 

\newtheorem{theorem}{Theorem}[section] 
\newtheorem{lemma}[theorem]{Lemma}     
\newtheorem{corollary}[theorem]{Corollary}
\newtheorem{proposition}[theorem]{Proposition}

\newtheorem*{definition}{Definition}
\newtheorem*{example}{Example}

\newcommand{\ch}{\mbox{\it\c C}\,}

\newcommand{\F}{\mathcal F}

\newcommand{\Z}{\mathbf{Z}}

\renewcommand{\F}{\mathcal{F}}
\renewcommand{\H}{{\mathcal{H}}}
\newcommand{\M}{{\mathcal M}}

\newcommand{\R}{\mathbf R}

\newcommand{\psl}{\mathrm{PSL}_2 (\Z)}

\newcommand{\mat}[2]{\left (\begin{array}{c c}
					#1 \\ #2
	          \end{array}
\right)}
\newcommand{\ccdot}{\!\cdot\!}

\newcommand{\aut}[1]{\mathrm{Aut} ({#1})}

 \newcommand{\sgn}{\mathop{\mathrm{sgn}}}

\newcommand{\tens}{{\otimes}}
\newcommand{\bull}{{\bullet}}

\newcommand{\oasterisk}{\bull}

\title{Binary quadratic forms as dessins} 

\author{A.Muhammed Uluda\u{g} and Ayberk Zeytin}
\address{Department of Mathematics Galatasaray University, Turkey}

\author{Merve Durmu\c{s}}
\address{Department of Mathematics Yeditepe University, Turkey}

\email{muludag{@}gsu.edu.tr, azeytin@gsu.edu.tr and merve1988durmus@gmail.com}


\begin{document}

\maketitle

\begin{abstract}
	We show that the class of every primitive indefinite binary quadratic form is naturally represented by an infinite graph  (named \c{c}ark) with a unique cycle embedded on a conformal annulus. This cycle is called the spine of the \c cark. Every choice of an edge of a fixed \c{c}ark specifies an indefinite binary quadratic form in the class represented by the \c{c}ark. Reduced forms in the class represented by a \c{c}ark correspond to some distinguished edges on its spine. Gauss reduction is the process of moving the edge in the direction of the spine of the \c{c}ark. Ambiguous and reciprocal classes are represented by \c{c}arks with symmetries. Periodic \c{c}arks represent classes of non-primitive forms. 
\end{abstract}


\section{Introduction.}
The Euclidean algorithm is the process of comparison of commensurable magnitudes and the modular group $\psl$ is an encoding of this algorithm. Since the intellect is ultimately about comparison of magnitudes, it should come as no surprise that the modular group manifests itself in diverse contexts through its action on mathematical objects, no matter what our level of abstraction is. Among all manifestations of $\psl$ the following four classical actions are of fundamental nature:
 
\begin{itemize}
	\item [1.] its left-action on the infinite trivalent plane tree, 
	\item [2.] its left action on the upper half plane $\H$ by M\"obius transformations, 
	\item [3.] its right-action on the binary quadratic forms, and
	\item [4.] its left-conjugation action on itself.
\end{itemize}

Our aim in this paper is to clarify the connections between these four actions. See \cite{uludag/actions/of/the/modular/group} for an overview of the related subjects from a wider perspective. In particular, the actions in consideration will play a crucial role in observing non-trivial relations between Teichm\"{u}ller theory and arithmetic. Such a point of view will be taken in a forthcoming paper where we construct a global groupoid whose objects are (roughly speaking) all ideal classes in real quadratic number fields and morphisms correspond to basic graph transformations known as flips. And this work can also be considered as an introduction to this upcoming work. 

Let us turn back to our list of actions. The first one is transitive but not free on the set of neither edges nor vertices of the tree in question. In order to make it free on the set of edges, we add the midpoints as extra vertices thereby doubling the set of edges and call the resulting infinite  tree the {\it bipartite Farey tree $\F$}. The modular group action is still transitive on the edge set of $\F$. Now since $\psl$ acts on $\F$ by automorphisms; freely on the set of edges of $\F$, so does any subgroup $\Gamma$ of $\psl$, and by our definition a {\it modular graph}\footnote{Contributing to the long list of names and equivalent/dual notions with various nuisances: trivalent diagrams, cyclic trivalent graphs, cuboid tree diagrams, Jacobi diagrams, trivalent ribbon graphs, triangulations; more generally, maps, ribbon graphs, fat graphs, dessins, polygonal decompositions, lozenge tilings, coset diagrams, etc.} is simply a quotient graph $\Gamma\backslash\F$. This is almost the same thing as a trivalent ribbon graph, except that we consider the midpoints as extra 2-valent vertices and pending edges are allowed. Modular graphs parametrize subgroups up to conjugacy and modular graphs with a base edge classify subgroups of the modular group.

\begin{figure}[h!]
	\centering
	\begin{subfigure}[]{5.5cm}
		\centering
		\includegraphics[scale=0.35]{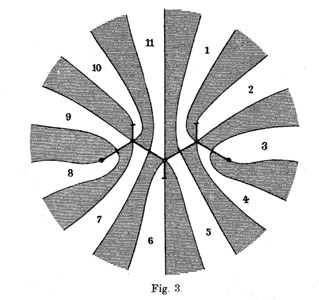}
		\caption{\scriptsize  A dessin  (linienzug of Klein) from 1879 \protect\cite{klein/linienzuge}}
		\label{fig:linienzuge/of/klein}
	\end{subfigure}
	\qquad
	\begin{subfigure}[]{4cm}
		\centering
		\vspace{1mm}
		\includegraphics[scale=0.2]{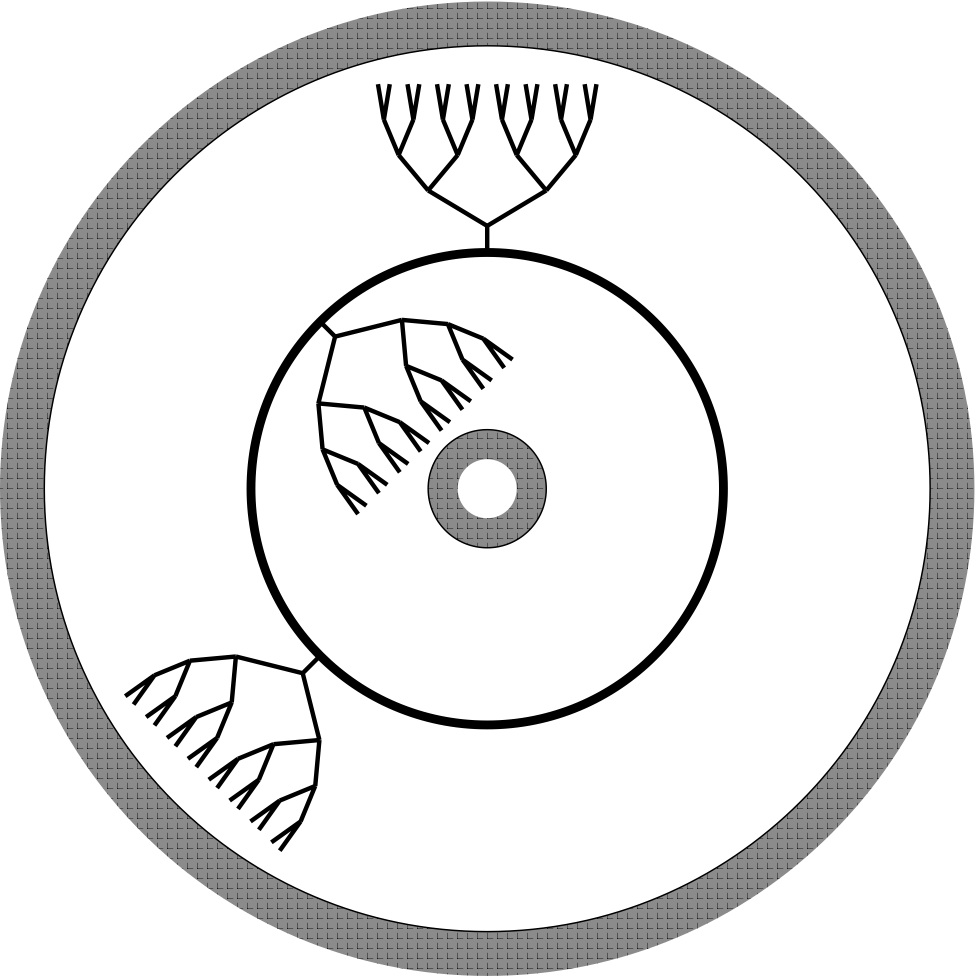}
		\vspace{4mm}
		\caption{\scriptsize A \c{c}ark in its ambient annulus}
		\label{fig:an/example/of/a/cark}
	\end{subfigure}
	\caption{}
	\label{fig:cark/and/its/short/form}
\end{figure}

The second action is compatible with the first one in the following sense: The tree $\F_{top}\subset \H$ which is built as the $\psl$-orbit of the arc connecting two elliptic points on the boundary of the standard fundamental domain, is a topological realization of the Farey tree $\F$. Consequently, $\Gamma\backslash\F_{top} \subset \Gamma\backslash\H$ is a topological realization of the graph $\Gamma\backslash\F$, as a graph embedded in the orbifold $\Gamma\backslash\H$. This latter has no orbifold points if  $\Gamma$ is torsion-free but always has punctures due to the parabolic elements of $\Gamma$, or it has some boundary components. These punctures are in one-to-one correspondence with the left-turn circuits in $\Gamma\backslash\F$. Widening these punctures gives a deformation retract of the ambient orbifold to the graph, in particular the upper half plane $\H$ retracts to the Farey tree $\F_{top}$. To recover the orbifold from the modular graph one glues punctured discs along the left-turn paths of the graph.

If $\Gamma$ is torsion-free of finite index, then $\Gamma\backslash\H$ is an algebraic curve which can be defined over a number field since it is a finite covering of the modular curve ${\mathcal M} = \psl \backslash \H$. According to Bely\u{\i}'s theorem, \cite{belyi}, any arithmetic surface can be defined this way, implying in particular that the action of the absolute Galois group defined on the set of finite coverings $\{ \Gamma\backslash\H\rightarrow {\mathcal M}\}$ is faithful. But these coverings are equivalently described by the graphs $\Gamma\backslash\F$. This striking correspondence between combinatorics and arithmetic led Grothendieck to study dessins from the point of view of the action of the absolute Galois group, see \cite{lando-zvonkine-lowdimtop}. However, explicit computations of covering maps $\Gamma\backslash \H\rightarrow {\mathcal M}$ required by this approach turned out to be forbiddingly hard if one wants to go beyond some basic cases and only a few uniform theorems could be obtained. In fact, dessins are more general graphs that correspond to finite coverings of the thrice punctured sphere, which is equivalent to a subsystem of coverings of  ${\mathcal M}$ since ${\mathbf P}^1\backslash\{0,1,\infty\}$ is a degree-6 covering of ${\mathcal M}$.

The third action in our list is due to Gauss. Here $\psl$ acts on the set of binary quadratic forms via change of variables in the well-known manner. Orbits of this action are called {\it classes} and forms in the same class are said to be {\it equivalent}. Here we are interested in the action on \emph{indefinite} forms. This action always has a cyclic stabilizer group, which is called the proper automorphism group of the form $f$ and denoted $\langle M_f \rangle$. Indefinite binary quadratic forms represent ideal classes in the quadratic number field having the same discriminant as the form and hence are tightly related to real quadratic number fields \cite{computational/nt/cohen}. We provide a succinct introduction to binary quadratic forms later in the paper. 

The correspondence between forms and dessins can be described briefly as follows: to an indefinite binary quadratic form $f$ we associate its proper automorphism group $\langle M_f \rangle$ and to $\langle M_f \rangle$ we associate the infinite graph $\langle M_f \rangle \backslash\F$, which is called a {\it \c{c}ark}\footnote{Turkish {\it \c cark}   (pronounced as ``chark'')  is borrowed from Persian, and it has a common etymology with Indian {\it chakra}, Greek {\it kyklos} and English {\it wheel}.}. Via the topological realization of $\F$, this is a graph embedded in the annulus $\langle M_f \rangle \backslash\H$. The form $f_M$ corresponding to the matrix $M\in \psl$ is found by homogenizing the fixed-point equation of $M$.
\c{C}arks are infinite ``transcendental" graphs whereas the dessins literature consider only finite graphs.  (``transcendental" since they correspond to non-algebraic extensions of the function field of the modular curve). This transcendence implies that \c carks go undetected in the algebraic fundamental group approach, nevertheless we shall see that this does not keep them away from being arithmetic objects. 

Equivalent forms have conjugate stabilizers  (automorphism groups) and conjugate subgroups have isomorphic quotient graphs. It turns out that the set of classes is exactly the set of orbits of hyperbolic elements of $\psl$ under the fourth  (conjugation) action in our list. This set of orbits can be identified with the set of bracelet diagrams with beads of two colors. 

In fact, \c carks can be thought of as $\Z$-quotients of periodic rivers of Conway \cite{conway/sensual/quadratic/form} or graphs dual to the coset diagrams of Mushtaq, \cite{mushtaq/modular/group}. As we shall see later in the paper, \c{c}arks provide a very nice reformulation of various concepts pertaining to indefinite binary quadratic forms, such as reduced forms and the reduction algorithm,  ambiguous forms, reciprocal forms, the Markoff value of a form, etc. For example, \c{c}arks of reciprocal classes admit an involutive automorphism, and the quotient graph gives an infinite graph with two pending edges. These graphs parametrize conjugacy classes of dihedral subgroups of the modular group. \c{C}arks also provide a more conceptual way to understand the relation between coset diagrams and quadratic irrationalities and their properties as studied in \cite{mushtaq/modular/group} or in \cite{malik/zafar/real/quadratic/irrational/numbers}.

\begin{figure}[h!]
	\centering
	 \includegraphics[scale=0.50]{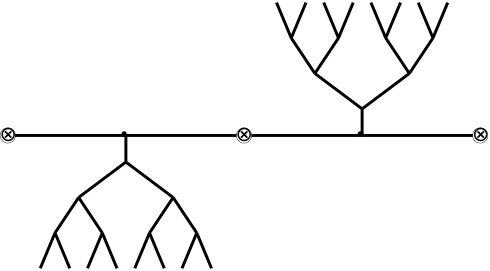}
	\caption{The graph of a reciprocal class. Edges of this graph parametrize the reciprocal forms in this class.}
	\label{fig:recip}
\end{figure}

For us the importance of this correspondence between \c carks and forms lies in that it suggests a concrete and clear way to consider modular graphs as arithmetic objects viz. Gauss’ binary quadratic forms, as it was much solicited by Grothendieck’s dessins school. We also would like to remark that the clarity of the graph language provides us with new points of view on the classical and deep questions concerning the behavior of class numbers yet the structure of class groups via such graphs remain meager. Moreover, the second named author, in \cite{reduction}, has presented an improvement of the age-old reduction algorithm of Gauss and gave an algorithmic solution to the representation problem of binary quadratic forms. The language of \c{c}arks might also provide a new insight o the real multiplication project of Manin and Marcoli, see \cite{manin/real/multiplication}.

Our computations concerning forms and their reduction are done in PARI/GP, \cite{PARI2} with certain subroutines of our own (source code is available upon request). 

\section{Farey tree and modular graphs}
\label{sec:farey/tree/and/modular/graphs}

It is well known that the two elliptic transformations $S (z)=-1/z$ and $R (z)= (z-1)/z$, respectively of orders 2 and 3, generate a group of M\"obius transformations which is isomorphic to the projective group of two by two integral matrices having determinant $1$, the modular group 
\cite{bqf/vollmer}. It is also well-known that  $\psl \cong \langle S\rangle \ast \langle R\rangle =\Z/2\Z \ast \Z / 3\Z$. Let us now consider the graph $\F$  (the bipartite Farey tree), given by the following data:
\begin{eqnarray*}
	E  (\F) & = & \{\{W\} \colon W \in \psl \} \\
	V  (\F) & = & V_{{\tens}}  (\F) \sqcup V_{\oasterisk}  (\F);
\end{eqnarray*}
\noindent where 
\begin{eqnarray*}
	V_{\tens}  (\F) & = & \{ \{W,WS\} \colon W \in \psl \}, \\
	V_{\oasterisk}  (\F) & = & \{ \{W,WR,WR^{2}\} \colon W \in \psl \}.
\end{eqnarray*}

\noindent  is an edge between a vertex $v=\{W,WS\} \in V_{\tens}  (\F)$ and another vertex $v'=\{W',W'R,W'R^{2}\}$ if and only if $\{W,WS\} \cap \{W',W'R,W'R^{2}\} \neq \emptyset$ and there are no other edges. Thus the edge connecting $v$ and $v'$ {\it is} $v\cap v'$, if this intersection is non-empty. Observe that by construction the graph is bipartite. The edges incident to the vertex $\{W,WR,WR^{2}\} \in V_{\oasterisk}$ are $ \{W\}, \{WR\}, \{WR^{2}\} $, and these edges inherit a natural cyclic ordering from the vertex. Thus the Farey tree $\F$ is an infinite bipartite ribbon graph\footnote{A ribbon graph is a graph together with an ordering of the edges that are incident to each vertex in the graph.}. It is a tree since $\psl$ is freely generated by $S$ and $R$.

The group $\psl$ acts on $\F$ from the left, by ribbon graph automorphisms, where $M\in \psl$ acts by
\begin{eqnarray*}
\{W\}\in E (\F) 								& \mapsto & \{MW\}\in E (\F)\\
\{W,WS\}\in  V_{{\tens}} (\F) 			& \mapsto & \{MW,MWS\}\in V_{{\tens}} (\F)\\
\{W,WR,WR^2\} \in V_{\oasterisk} (\F) & \mapsto & \{MW,MWR,MWR^2\}\in V_{\oasterisk} (\F)
\end{eqnarray*}
Notice that the action on the set of edges is nothing but the left-regular action of $\psl$ on itself and therefore is free. 
On the other hand the action is not free on the set of vertices: The vertex $\{W,WS\}$ is fixed by the order-2 subgroup generated by $M=WSW^{-1}$, 
and the vertex $\{W,WR,WR^2\}$ is fixed by the order-3 subgroup generated by $M=WRW^{-1}$.

Let $\Gamma$ be any subgroup of $\psl$. Then $\Gamma$ acts on $\F$ from the left and to $\Gamma$ we associate a quotient graph $\Gamma\backslash\F$ as follows:

\medskip\noindent
\hspace{1.5cm}	 $E  (\Gamma\backslash\F) = \{\Gamma \ccdot \{W\} \colon W \in \psl \}$

\noindent
\hspace{1.5cm}	 $V  (\Gamma\backslash\F) = V_{\tens}  (\F/\Gamma) \cup V_{\oasterisk}  (\F/\Gamma)$;

\noindent where 

\noindent
\hspace{1.5cm}	$V_{\tens}  (\Gamma\backslash\F) = \{ \Gamma \ccdot\{W , WS\} \colon W \in \psl \}$, and

\noindent
\hspace{1.5cm}	$V_{\oasterisk}  (\Gamma\backslash\F) = \{ \Gamma \ccdot\{W, WR, WR^{2}\} \colon W \in \psl \}$.

\medskip\noindent It is easy to see that the incidence relation induced from the Farey tree gives a well-defined incidence relation and gives us the  graph which we call a {\it modular graph}. Thus the edge connecting the vertices $v=\Gamma\ccdot \{W,WS\}$ and $v'=\Gamma\ccdot\{W',W'R,W'R^{2}\}$ is the intersection $v \sqcap v'$, which is of the form $\Gamma\ccdot \{M\}$  if non-empty. There are no other edges. Observe that by construction the graph is bipartite. The edges incident to the vertex $\Gamma\ccdot \{W,WR,WR^{2}\}$ are $ \Gamma\ccdot \{W\}, \Gamma\ccdot \{WR\}, \Gamma\ccdot \{WR^{2}\} $, and these edges inherit a natural cyclic ordering from the vertex\footnote{The ribbon graph structure around vertices of degree 2 is trivial.}. In general $\Gamma\backslash\F$ is a bipartite ribbon graph possibly with pending vertices that corresponds to the conjugacy classes of elliptic elements that $\Gamma$ contains. Conversely, any connected bipartite ribbon graph $G$, with $V (G)=V_{\tens}  (G) \sqcup V_{\oasterisk}  (G)$, such that every $\tens$-vertex is of degree 1 or 2 and every $\oasterisk$-vertex is of degree 1 or 3, is modular since the universal covering of $G$ is isomorphic to $\F$. It takes a little effort to define the fundamental group of $\Gamma\backslash\F$, see \cite{nedela/graphs/and/their/coverings}, so that there is a canonical isomorphism $\pi_1 (\Gamma\backslash\F, \Gamma\ccdot \{I\})\simeq \Gamma<\psl$, with the canonical choice of $\Gamma\ccdot \{I\}$ as a base edge. In general, subgroups $\Gamma$ of the modular group  (or equivalently the fundamental groups $\pi_1 (\Gamma\backslash\F)$) are free products of copies of $\Z$, $\Z/2\Z$ and $\Z/3Z$, see \cite{kulkarni/subgroups/of/the/modular/group}.  Note that two distinct isomorphic subgroups $\Gamma_1$, $\Gamma_2$ of the modular group may give rise to non-isomorphic ribbon graphs $\Gamma_1\backslash\F$ and $\Gamma_2\backslash\F$. We shall see shortly that \c{c}arks constitute good examples of this phenomena. In other words, the fundamental group does not characterize the graph. Another basic invariant of $\Gamma\backslash\F$ is its genus, which is defined to be the genus of the surface constructed by gluing discs along left-turn paths.  This genus is the same as the genus of the Riemann surface ${\mathcal H} / \Gamma$.

\begin{figure}[t!]
	\centering
	\begin{subfigure}[]{6cm}
		\centering
		\begin{tikzpicture} [scale=1.1]
			\draw [line width=0.75mm]  (0,0)-- (0,2)-- (4,0)-- (0,0);
			\fill[gray!40]  (0,0)-- (0,2)-- (4,0)-- (0,0);
			\draw  (0,0) circle  (0.8mm);
			\fill[white]  (0,0) circle  (0.8mm);
			\node at  (0,0) {$\otimes$};
			\draw  (4,0) circle  (0.8mm);
			\fill[white]  (4,0) circle  (0.8mm);
			\draw  (0,2) circle  (0.8mm);
			\fill[black]  (0,2) circle  (0.8mm);
			\draw  (-0.05,1)-- (-0.25,1);
			\node at  (-0.25,0.99) {\scriptsize $<$};
			\node at  (0.25,1)[rotate=90] {\scriptsize{modular arc}};
			\node at  (-0.75,1)[rotate=0] {\scriptsize{modular}};
			\node at  (-0.75,0.75)[rotate=0] {\scriptsize{graph}};
			\draw  (2,-0.05)-- (2,-0.25);
			\node at  (2,-0.25)[rotate=90] {\scriptsize $<$};
			\node at  (2,-0.5)[rotate=0] {\scriptsize{triangulations}};
			\draw [rotate around={0: (0,0)}] (2.05,1)-- (2.2,1.3);
			\node at  (2.2,1.3) [rotate=70]{\scriptsize $>$};
			\node at  (2.25,1.5)[rotate=-30] {\scriptsize{lozenges}};
		\end{tikzpicture}
		\caption{The fundamental region for the modular curve in the upper half plane model.}
	\end{subfigure}
	\begin{subfigure}[]{6cm}
		\centering
		\begin{tikzpicture}  [scale=1.1,dash pattern=on 2pt off 1pt]
			\fill [gray!40]  (-0.5,0) rectangle  (0.5,2);
			{ \pgfsetfillopacity{1}
			\fill[white]  (0:0cm) circle  (1cm);}
			\draw [solid,line width=1] (-2,0)-- (2,0);
			\draw [dash phase=1pt]  (-0.5,0)-- (-0.5,2);
			\draw [solid,line width=1]  (0.5,cos{30})-- (0.5,2);
			\draw [solid,line width=1]  (-0.5,cos{30})-- (-0.5,2);
			\draw [dash phase=1pt]  (0.5,0)-- (0.5,2);
			\draw [dash phase=1pt]  (0,1)-- (0,2.5);
			\node at  (0,2.5)[rotate=-90] {\scriptsize $<$};
			\draw [solid,line width=1.25]  (1,0) arc  (0:180:1);
			\node at  (-0.5,-0.2) {\scriptsize $-1/2$};
			\node at  (0.5,-0.2) {\scriptsize $1/2$};
			\draw  (0,1) circle  (0.75mm);
			\draw  (cos{60},sin{60}) circle  (0.75mm);
			\fill[white]  (cos{60},sin{60}) circle  (0.75mm);
			\draw  (-cos{60},sin{60}) circle  (0.75mm);
			\fill[white]  (-cos{60},sin{60}) circle  (0.75mm);
			\fill[white]  (0,1) circle  (0.75mm);
			\node at   (0,1) {\small $\otimes$};
			\fill[black]  (cos{60},sin{60})  circle  (0.75mm);
			\fill[black]  (-cos{60},sin{60}) circle  (0.75mm);
		\end{tikzpicture}
		\caption{The modular curve. Note that there are two triangles, the second is on the back of the page, glued to this one.}
	\end{subfigure}
	\caption{}
	\label{fig:cark/and/its/short/form}
\end{figure}

The set of edges of $\Gamma\backslash\F$ is identified with the set of right-cosets of $\Gamma$, so that the graph $\Gamma\backslash\F$ has $[\psl : \Gamma]$ many edges. In case $\Gamma$ is a finite index subgroup, the graph $\Gamma\backslash\F$ is finite. In case $\Gamma = \psl$, the quotient graph $\psl \backslash \F$ is a graph with one edge that looks like as follows:

\begin{figure}[h!]
	$$ \stackrel{\psl\ccdot\{I,S\}}{{\otimes}}\!\!\!\!\!\!\!\!\!\!\!\!\!\hspace{-.1mm}\stackrel{\psl\ccdot \{I\}}{-\!\!\!-\!\!\!-\!\!\!-\!\!\!-\!\!\!-\!\!\!-\!\!\!-\!\!\!-\!\!\!-\!\!\!-\!\!\!-\!\!\!-\!\!\!-\!\!\!-\!\!\!-\!\!\!-\!\!\!-\!\!\!-\!\!\!-\!\!\!-\!\!\!-\!\!\!-\!\!\!-\!\!\!-\!\!\!-\!\!\!-\!\!\!-\!\!\!}\!\!\!\!\!\!\!\!\!\!\!\!\!\!\!\stackrel{\psl\ccdot\{I,R,R^2\}}{{   
	\mbox{\raisebox{-.2mm}{\Large $\bullet$}}   }}$$
	\caption{ The modular arc.}
	\label{fig:modular/arc}
\end{figure}

We call this graph the {\it modular arc}. It is a graph whose fundamental group is $\psl$ and whose universal cover is the Farey tree $\F$. In other words modular graphs are coverings of the modular arc. If we consider the action of the modular group on the topological realization $\F_{top}$ of $\F$ mentioned in the introduction, the topological realization of $\psl\backslash\F$ is the arc $\psl\backslash\F_{top}$ in the modular curve connecting two elliptic points.

Every modular graph $\Gamma\backslash \F$ has a canonical ``analytical" realization $\Gamma\backslash \F_{top}$ on the Riemann surface $\Gamma \backslash {\mathcal H}$ with edges being geodesic segments. Equivalently, these edges are lifts of the modular arc by $\Gamma \backslash {\mathcal H} \longrightarrow \psl\backslash {\mathcal H}$. If instead we lift the geodesic arc connecting the $\tens$- elliptic point to the cusp to the surface $\Gamma \backslash {\mathcal H}$, then we obtain another graph on the surface, which is called an {\it ideal triangulation}. Lifting the remaining geodesic arc gives rise to yet another type of graph, called a {\it lozenge tiling}. So there is a triality, not just duality, of these graphs, see Figure~\ref{fig:triality} in which the bold figures represent the members of the triality.

\begin{figure}[h!]
	\centering
	\begin{subfigure}{3cm}
		\includegraphics[scale=0.18]{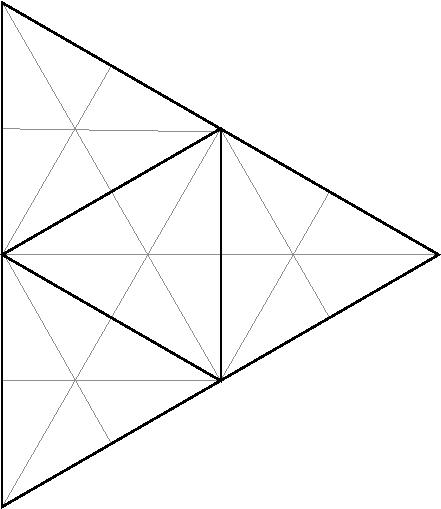}		
		\caption{A triangulation}
	\end{subfigure}
	\begin{subfigure}{3cm}
		\includegraphics[scale=0.18]{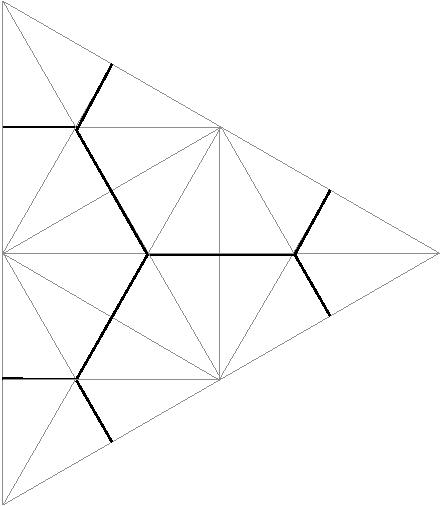}		
		\caption{its dual graph}
	\end{subfigure}
	\begin{subfigure}{3cm}
		\includegraphics[scale=0.18]{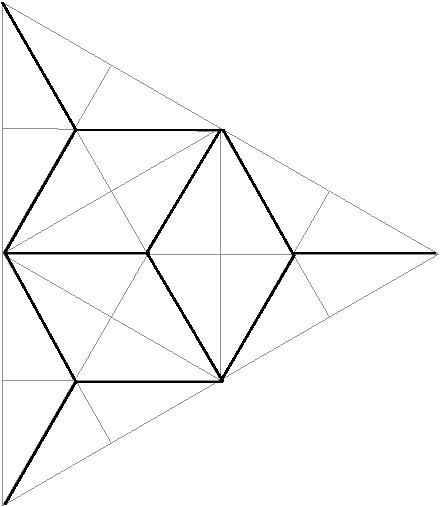}
		\caption{and its lozenge}
	\end{subfigure}
	
	\caption{ Triality of graphs}
	\label{fig:triality}
\end{figure}

In topology, there is a well-known correspondence between subgroups of the fundamental group of a space and the coverings of that space. The following two results are orbifold  (or ``stacky") analogues of this correspondence for coverings of the modular curve, stated in terms of graphs. For more details on fundamental groups and covering theory of graphs see \cite{nedela/graphs/and/their/coverings}.

\begin{proposition}
	If $\Gamma_1$ and $\Gamma_2$ are conjugate subgroups of $\psl$, then the graphs $\Gamma_1\backslash\F$ and $\Gamma_2\backslash\F$ are isomorphic as ribbon graphs. Hence there is a 1-1 correspondence between modular graphs and conjugacy classes of subgroups of the modular group. 
	\label{conjugatesubgroups}
\end{proposition}
\begin{proof}
	Let $\Gamma_2=M\Gamma_1M^{-1}$. The desired isomorphism is then the map
	\begin{eqnarray*}
		\varphi:E (\Gamma_1\backslash\F) 	& \rightarrow E (\Gamma_2\backslash\F) \\
						\Gamma_1\ccdot \{W\}		& \,\,\, \mapsto \Gamma_2\ccdot \{MW\}.
	\end{eqnarray*}
	\noindent Note that one has $\varphi (\Gamma_1\ccdot \{I\})=\Gamma_2\ccdot \{M\}$. Suppose now that $\varphi: E (\Gamma_1\backslash\F) \rightarrow E (\Gamma_2\backslash\F)$ is a ribbon graph isomorphism and let $\varphi (\Gamma_1\ccdot \{I\})=\Gamma_2\ccdot \{M\}$. This induces an isomorphism of fundamental groups 

		$$\varphi_*:\pi_1 (\Gamma_1\backslash\F, \Gamma_1\ccdot \{I\}) \simeq \pi_1 (\Gamma_2\backslash\F, \Gamma_2\ccdot \{M\})$$

	\noindent Since $\varphi$ is a ribbon graph isomorphism, these two groups are also isomorphic as subgroups of the modular group. The former group is canonically isomorphic to $\Gamma_1$ a whereas the latter group is canonically isomorphic to 
	$$M^{-1}\pi_1 (\Gamma_2\backslash\F, \Gamma_2\ccdot \{I\})M\simeq M^{-1}\Gamma_2 M$$
\end{proof}

Therefore modular graphs parametrize conjugacy classes of subgroups of the modular group, whereas the edges of a modular graph parametrize subgroups in the conjugacy class represented by the modular graph. In conclusion we get:

\begin{theorem} 
	There is a 1-1 correspondence between modular graphs with a base edge $(G,e)$  (modulo ribbon graph isomorphisms of pairs $ (G,e)$) and subgroups of the modular group (modulo the automorphisms induced by conjugation in $\psl$). 
	\label{thm:modular/graph/vs/subgroups}
\end{theorem}

\begin{theorem} 
	{There is a 1-1 correspondence between modular graphs with two base edges $ (G,e,e')$  (modulo ribbon graph isomorphisms of pairs $ (G,e,e')$) and cosets of subgroups of the modular group ((modulo the automorphisms induced by conjugation in $\psl$)). 
	\label{thm:modular/graph/vs/cosets}}
\end{theorem}

\section{\c{C}arks}
\label{sec:carks}

A {\it \c{c}ark} is a modular graph of the form $\ch_M:=\langle M \rangle \backslash \F$ where $M$ is a hyperbolic element of the modular group. One has 
$$
	\pi_1 (\langle M \rangle\backslash \F)=\langle M \rangle\simeq \Z,
$$ 
so the \c{c}ark $\langle W \rangle\backslash \F$ is a graph with only one circuit, which we call the {\it spine} of the \c{c}ark. Every \c{c}ark has a canonical realization as a graph $\langle M \rangle \backslash \F_{top}$ embedded in the surface $\langle M \rangle \backslash \H$, which is an annulus since $M$ is hyperbolic. In fact $\langle M \rangle \backslash \H$ is the annular uniformization of the modular curve $\M$ corresponding to $M\in\pi_1 (\M)$. Again by hyperbolicity of $M$, this graph will have infinite ``Farey branches" attached to the spine in the direction of both of the boundary components of the annulus\footnote{If $M$ is parabolic, then $\langle W \rangle\backslash \F$ has Farey branches attached to the spine in only one direction, and its topological realization $\langle M \rangle \backslash \F_{top}$ sits on a punctured disc. If $M$ is elliptic, $\langle W \rangle\backslash \F$ is a tree with a pending edge which abut at a vertex of type ${\otimes}$ when $M$ is of order 2 and of type $\bullet$ when $M$ is of order 3. Its topological realization $\langle M \rangle \backslash \F_{top}$ sits on a disc with an orbifold point.}. By Proposition~\ref{conjugatesubgroups} the graphs $\ch_M$ and $\ch_{XMX^{-1}}$ are isomorphic for every element $X$ of the modular group and by Theorem~\ref{thm:modular/graph/vs/subgroups} we deduce the following result, see \cite{merve/tez}:


\begin{theorem}  
	There are one-to-one correspondences between
	\begin{itemize}
		\item[i.] \c{c}arks and conjugacy classes of subgroups of the modular group generated by a single hyperbolic element, and
		\item[ii.] \c{c}arks with a base edge and subgroups of the modular group generated by a single hyperbolic element.
	\end{itemize}
	\label{cor:1/to/1/corr/between/undirected/carks}
\end{theorem}

\begin{figure}[h!]
	\centering
	 \includegraphics[scale=0.23]{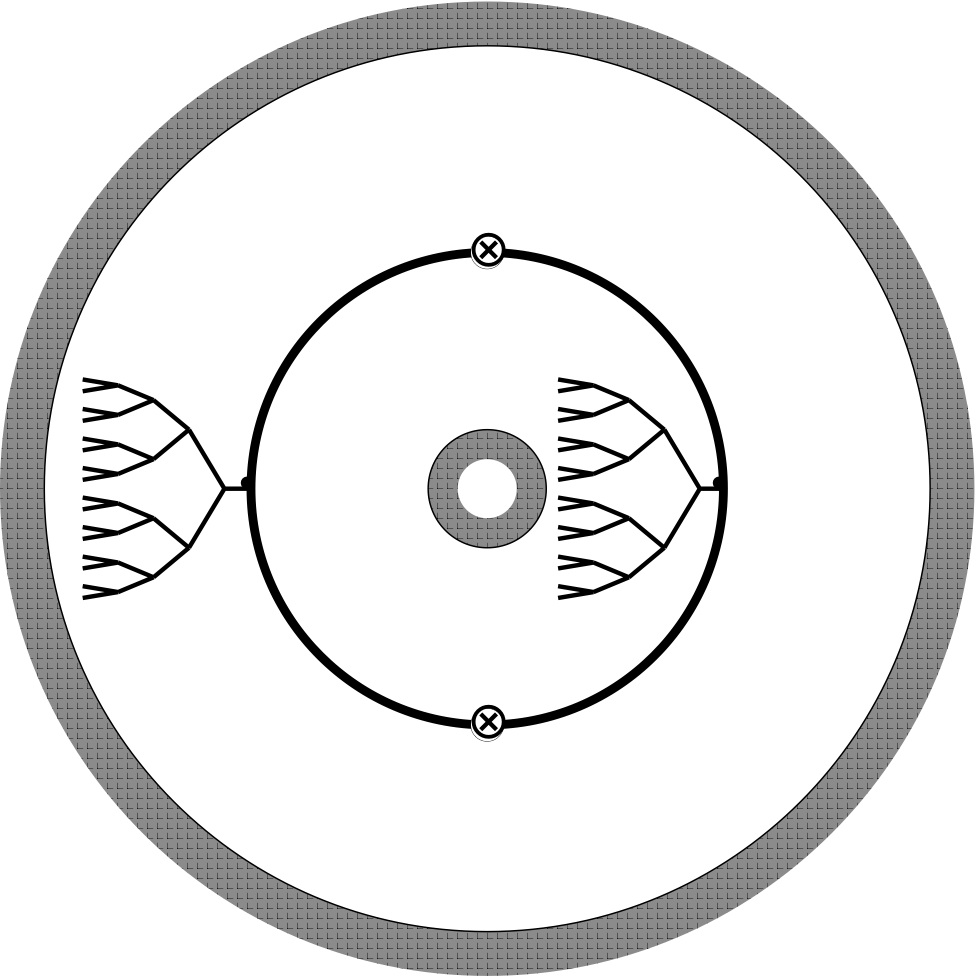}
	\caption{ The \c cark $\F/\langle S R^2 S R \rangle$.}
	\label{fig:quotient/by/hyperbolic}
\end{figure}

A \c{c}ark is said to be directed if we choose an orientation for the spine. 

\begin{corollary} 
	There are one-to-one correspondences between
	\begin{itemize}
		\item[i.] hyperbolic elements of the modular group and directed \c carks with a base edge, and
		\item[ii.] conjugacy classes of hyperbolic elements of the modular group and directed \c carks.
	\end{itemize}
	\label{cor:1/to/1/corr/between/directed/carks}
\end{corollary}
\vspace{-5mm}
\subsection{Counting \c{C}arks}
\label{sec:necklace/bracelet} 
\c{C}arks are infinite graphs, and each edge of a \c{c}ark carries a name which is an infinite coset. In fact, all the combinatorial information of a \c cark can be encoded in a finite storage as follows: First remove all $\tens$-vertices of the \c{c}ark. Next, turn once around the spine. Upon meeting a $\bullet$-vertex on which a branch attached by $R$, cut that branch and tag that $\bullet$-vertex with a ``0". In a similar fashion, upon meeting a $\bullet$-vertex on which a branch attached by $R^{2}$, cut that branch and tag that $\bullet$-vertex with a ``1". We obtain a finite graph called a {\it binary bracelet} which is by definition an equivalence class of binary strings under cyclic permutations (i.e. rotations) and reversals. Conversely, by using the convention $0\leftrightarrow R$ and $1\leftrightarrow R^2$ we can reconstruct the \c{c}ark from its bracelet.

Rotations and reversals generate a finite dihedral group, and a binary bracelet may equivalently be described as an orbit of this action.

\begin{figure}[h!]
	\centering
	\begin{subfigure}{5cm}
		\includegraphics[scale=0.35]{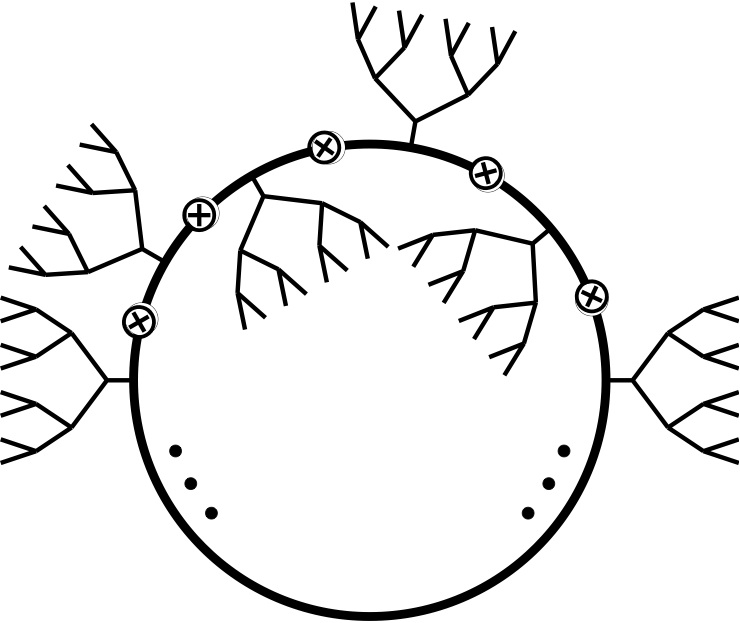}
		\caption{}
	\end{subfigure}
	\begin{subfigure}{5cm}
		\includegraphics[scale=0.43]{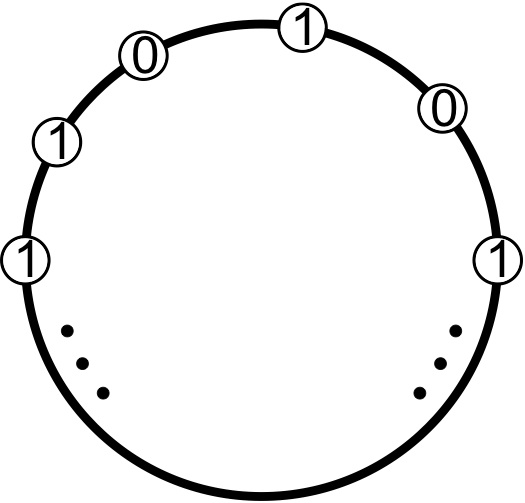}
		\caption{}
	\end{subfigure}
	
	\caption{From \c{c}arks to bracelets}
	\label{fig:cark/and/its/short/form}
\end{figure}

For $n = 1,2,...,15$ the number of binary bracelets with $n$ vertices is 
\[
2, 3, 4, 6, 8, 13, 18, 30, 46, 78, 126, 224, 380, 687, 1224. 
\]
This is sequence A000029 (M0563) in OEIS~\cite{oeis/integer/sequences}.
The number of binary bracelets  (\c{c}arks) of length $n$ is
\begin{eqnarray*}
	B (n) = {\frac{1}{2}}N (n) + {\frac{3}{4}}2^{n/2}
\end{eqnarray*}
if $n$ is even and 
\begin{eqnarray*}
	B(n) = {\frac{1}{2}}N (n) + {\frac{1}{2}}2^{ (n+1)/2}
\end{eqnarray*}
if $n$ is odd where $N (n)$ is the number of binary necklaces of length $n$.
An equivalence class of binary strings under rotations  (excluding thus reversals) is called a {\it binary necklace}, or a {\it cyclic binary word}.
They are thus orbits of words under the action of a cyclic group and they correspond to directed \c{c}arks.
For $n=1,2,...,15$ the number of  binary necklaces of length $n$ is
\[
N (n)=2, 3, 4, 6, 8, 14, 20, 36, 60, 108, 188, 352, 632, 1182, 2192,
\] 
which is sequence A000031 (M0564) in OEIS. 
The number of necklaces  (directed \c carks) of length $n$ is given by MacMahon's formula from 1892  (also called Witt's formula) reads (see \cite{bouallegue/on/primitive/words}, \cite{macmahon}):
 
\[
N (n)={1\over n}\sum_{d\mid n}\varphi (d)2^{n/d}={1\over n}\sum_{j=1}^n2^{\gcd (j,n)}
\]
where $\varphi$ is Euler's totient function.

A \c{c}ark is called {\it primitive} if its spine is not periodic. Aperodic binary necklaces correspond to primitive directed \c{c}arks. For $n=1,2,...,15$  the number of  aperiodic necklaces of length $n$ is
$$
L (n)=2, 1, 2, 3, 6, 9, 18, 30, 56, 99, 186, 335, 630, 1161, 2182,
$$
which is sequence A000031 (M0564) in the database. There is a formula for the number of aperiodic necklaces of length $n$ in terms of M\"obius' function $\mu$:
$$
L (n)={1\over n}\sum_{d\mid n}\mu (d)2^{n/d}={1\over n}\sum_{d\mid n}\mu (n/d)2^{d}
$$

As mentioned, binary necklaces (or cyclic binary words or directed \c{c}arks) may be viewed as orbits of words under the action of the cyclic group. Choosing an ordering of our letters $\{0,1\}$ (i.e. $0<1$) and imposing the lexicographic ordering of the words, one may choose a minimal representative in each orbit. The minimal representative of a primitive (aperiodic) word is called a {\it Lyndon word}. They were first studied in connection with the construction of bases for free lie algebras and they appear in numerous contexts. In our case they are
$$
0,1,01,001,011, 0001, 0011, 0111, 00001,00011,00101, 00111,01011,01111 \dots
$$

One can similarly find representatives for aperiodic binary bracelets (=primitive indefinite binary quadratic forms; see below). There are effective algorithms to list all primitive necklaces and bracelets up to a given length 
(i.e. Duval's algorithm \cite{duval/classes/de/conjugation},  the algorithm due to Fredricksen, Kessler and Maiorana \cite{FKM}, Sawada's algorithm \cite{sawada}, etc). Translated into the language of binary quadratic forms, this means that it is possible to single out a unique reduced representative in each class of a primitive indefinite binary quadratic form and that it is possible to effectively enumerate all classes of primitive indefinite binary quadratic forms by specifying those reduced representatives.

To sum up, we may represent primitive \c carks by primitive bracelets. In order to shorten this representation further, 
we may count the number of consecutive 0's and 1's and represent \c carks as sequences of natural numbers
$(n_0, n_1, \dots n_{2k})^{0,1}$, if we agree that\footnote{Note that a Lyndon word always start with a 0 and ends with a 1.} this sequence encodes a bracelet that starts with a $0$ if the exponent is $0$ and $1$ if the exponent is  $1$. This representation is directly connected to the ``minus" continued fractions (see Zagier \cite{zagier/zetafunktionen/quadratische/zahlkorper}).

A primitive word may have two types of symmetries: invariance under the swap of symbols $0 \leftrightarrow 1$ and invariance under reversal, i.e. palindromic symmetry. The first symmetry corresponds to ambiguous binary quadratic forms and the second symmetry corresponds to reciprocal binary quadratic forms, as we shall see. The swap of symbols $0 \leftrightarrow 1$ corresponds to inversion in the class group.

\subsection{\c Cark Invariants}
There are several natural invariants associated to a \c cark {\it \c C}. The combinatorial length $l_c (\ch)$ of its spine is an invariant. A hyperbolic invariant of a \c cark is the metric length $l_h (\ch)$  of the closed geodesic in the annular surface under its hyperbolic metric induced by the \c cark. A conformal invariant of a \c cark is the modulus $m (\ch)$  of the associated annulus. Finally, the discriminant $\Delta (\ch)$ of the associated form and the absolute value of the trace $\tau (\ch)$ of the associated matrix are two arithmetic invariants with $\Delta=\tau^2-4$. One has
$$
l_h (\ch)=2 \mbox{ arccosh }  (\tau/2), 
\quad 
m (\ch)= \exp\left (
\frac{\pi^2}
{
\log |\frac{\tau\pm \sqrt{\Delta}}{2}|
}
\right)
$$
The modulus is found as follows: Any hyperbolic element $M\in PSL_2 (\R)$ is conjugate to an element of the form 
\[
N:=XMX^{-1}=\mat{\alpha & 0 }{0 & \frac{1}{\alpha}}\]
where $\alpha$ is the multiplier of $M$. Since the trace is invariant under conjugation, one has
$\tau:=\mathrm{tr} (M)=\alpha+1/\alpha \Rightarrow \alpha^2-\tau\alpha+1=0 \Rightarrow \alpha=\frac{\tau\pm \sqrt{\tau^2-4}}{2}$.

Now $N$ acts by M\"obius transformation $z\mapsto \alpha^2 z$, and the quotient map is 
$f (z)=z^{2\pi i/\log \alpha^2}$ with the annulus $f ({\mathcal H})=\{z\, :\, e^{-2\pi^2/\log \alpha^2} < |z| <1\}$ as its image.
Hence the modulus of the ambient annulus of the \c cark is $e^{2\pi^2/\log \alpha^2}=e^{\pi^2/\log |\alpha|}$.
It is possible to write down the uniformization $U_M:\H\rightarrow \ch_M$ explicitly, which is a quite involved expression. 
The annular uniformization  $\ch_M\rightarrow \psl\backslash \H$ can be written as $j\circ U_M^{-}$.

\section{Binary Quadratic Forms and \c{C}arks}
\label{sec:bqfs/and/carks}
A \emph{binary quadratic form} is a homogeneous function of degree two in two variables $f  (x,y)=Ax^2+Bxy+Cy^2$   (denoted $f=  (A,B,C)$ or in the matrix form:
\begin{equation}
	W_{f} = \mat{A & B / 2}{B / 2 & C}
\end{equation}
so that $f  (x,y)=  (x,y)W_{f}  (x,y)^{t}$). If the coefficients $A,B,C$ are integers the form is called \emph{integral} with discriminant $\Delta  (f) = B^{2} - 4AC$. If $f$ is integral and $\gcd  (A,B,C)=1$ then $f$ is called \emph{primitive}. Following Gauss we will call a form $f=  (A,B,C)$ \emph{ambiguous} if $B=kA$ for some $k \in \Z$. Finally a form $f=  (A,B,C)$ will be referred to as \emph{reciprocal} whenever $C=-A$, \cite{sarnak/reciprocal/geodesics}.

Note that $\Delta  (f) = -4 \det  (W_{f})$. Given a symmetric two by two matrix we write $f_{W}$ to denote the binary quadratic form associated to $W$. Recall that a form $f$ is called
\begin{itemize}
	\item positive definite if and only if $\Delta  (f)< 0$ and $A>0$,
	\item negative definite if and only if $\Delta  (f)<0$ and $0>A$,
	\item indefinite if and only if $\Delta  (f)>0$. 
\end{itemize}

The group $\psl$ acts on the set of all integral binary quadratic forms by
\begin{eqnarray*}
	\mathit{Forms} \times \psl		\to  & \mathit{Forms}	& \\
	  (f,U) 							\mapsto 	& {U \ccdot f} := f  (U  (x,y)^{t}) \\
													& 	 =   (x,y)U^{t}W_{f}U  (x,y)^{t}
\end{eqnarray*}

We call two binary quadratic forms \emph{equivalent} if they belong to the same $\psl$ orbit under the above action, under which discriminant is invariant. Let us denote the $\psl$-orbit   (or the equivalence class) of $f$ by $[f]$. The stabilizer  of $f$ is called its {\it automorphism group}, denoted by $\aut{f}$, and elements of $\aut{f}$ are called automorphisms of $f$. For a positive definite binary quadratic form $f$, the group $\aut{f}$ is trivial unless $\Delta  (f) = -3 $ or $-4$;  $\aut{f} \simeq \Z / 4\Z$ if $\Delta  (f) = -4$ and $\aut{f} \simeq \Z / 6\Z$ in case $\Delta  (f) = -3$, \cite[p.29]{bqf/vollmer}. On the other hand, for an indefinite binary quadratic form one has $\aut{f} \simeq \Z$.

Given an indefinite binary quadratic form $f=  (A,B,C)$ a generator of its automorphism group will be called its \emph{fundamental automorphism}. Note that there are two fundamental automorphisms, one being $M_{f}$, the other being its inverse, $M_{f}^{-1}$. Every integral solution $ (\alpha, \beta)$ of Pell's equation:
\begin{equation}
	X^{2} - \Delta(f) Y^{2} = + 4
	\label{eqn:Pell}
\end{equation}
corresponds to an automorphism of $f$ given by the matrix:
\begin{eqnarray*}
	\mat{\frac{\alpha - B \beta }{2} & -C \beta}{A \beta & \frac{\alpha + \beta B}{2}}.
\end{eqnarray*}
It turns out that the fundamental automorphism is the one having minimal $\beta$ \cite[Proposition 6.12.7]{bqf/vollmer}.

Conversely, to any given hyperbolic element, say $M = \mat{p&q}{r&s}\in \psl$ let\\[-1.6mm] us associate the following binary quadratic form:

\begin{equation}
	f_{M} =  \frac{\mathrm{sgn}  (p+s)  }{\textrm{gcd}  (q,s-p,r)}\bigl(r,  s-p,-q\bigr)
	\label{eq:from/matrices/to/bqfs}
\end{equation}

Observe first that $M\to f_{M}$ is well-defined and that its image is always primitive and indefinite.  At this point let us state a direct consequence of Theorem~\ref{thm:modular/graph/vs/subgroups}:
\begin{corollary}
	The maps  $\langle M \rangle\backslash \F \longleftrightarrow M\longrightarrow f_{M}$  defines a surjection from the set of oriented \c{c}arks with a base edge to primitive indefinite binary quadratic forms.
	\label{cor:oriented/carks/with/a/base/edge/vs/bqf}
\end{corollary}

\begin{proof}
	We saw that an oriented \c{c}ark with a base edge determines a hyperbolic element of $\psl$. And this element in turn determines an indefinite binary quadratic form via  $M\to f_{M}$. Conversely, given a primitive indefinite binary quadratic form $f =   (A,B,C)$ to find $\beta \in \Z$ such that the matrix 
	\begin{eqnarray*}
		\mat{\beta & A}{-C & B + \beta} \in \psl
	\end{eqnarray*}
	we look at solutions $(x,y)$ of Pell's equation $X^{2} - \Delta  (f)Y^{2} = 4$. Using any such $y$ we construct the hyperbolic element:
	\begin{eqnarray*}
		M_{f} = \mat{\beta & y C}{y A & yB + \beta},
	\end{eqnarray*}
	where $\beta = \frac{-y B \pm x}{2}$. Both choices of the sign produces a matrix which maps onto $f$. In fact, the two matrices are inverses of each other in $\psl$.
\end{proof}

\begin{example}
	Consider the form $  (1,7,-1)$. It has discriminant $53$. The pair $  (51,7)$ is a solution to the Pell equation $X^{2} - 53Y^{2} = 4$. The two $\beta$ values corresponding to this solution are $-50$ and $1$. Plugging these two values into the matrix above we get:
	\begin{eqnarray*}
		M_{o} = \mat{1 & 7}{7 & 50} \mbox{ and } M_{o}^{-1}= \mat{-50 & 7}{7 & -1}.
	\end{eqnarray*}
	The pair $  (2599,357)$ is also a solution to the above Pell equation, and the corresponding matrices are:
	\begin{eqnarray*}
		\mat{50 & 357}{357 & 2549} \mbox{ and } \mat{-2549 & 357}{357 & -50}.
	\end{eqnarray*}
	We would like to remark also that
	\begin{eqnarray*}
		M_{o}^{2} = \mat{50 & 357}{357 & 2549}.
	\end{eqnarray*}
	In fact, $M_{o}$ is one of the two fundamental automorphisms of $f$.
\end{example}

Note that the map $W \mapsto f_{W}$ is infinite to one because any indefinite binary quadratic form has infinite automorphism group. Any matrix in the automorphism group of $f$ maps onto $f$.

Let $\mathcal{D}:=\{ d\in \Z_{>0} \colon d \equiv 0,1   \,\, (\mbox{mod }4), \,  d \mbox{ is not a square}\}$. Recall the following:

\begin{proposition}[{\cite{sarnak/reciprocal/geodesics}}]
	There is a bijection between the set of conjugacy classes of primitive hyperbolic elements in $\psl$ and the set of classes of primitive binary quadratic forms of discriminant $\Delta \in \mathcal{D}$; where a hyperbolic element is called primitive if it is not a power of another hyperbolic element.
\end{proposition}

\subsection{Reduction Theory of Binary Quadratic Forms} 
\label{sec:reducetion/theory}

We say that an indefinite binary quadratic form $f= (A,B,C)$ is {\it reduced} if the geodesic in $\mathcal{H}$ connecting the two real fixed points of $W_{f}$, called the {\it axis} of $W_{f}$ and denoted by $\mathfrak{a}_{W_{f}}$, intersects with the standard fundamental domain of the modular group. Remark that this definition is equivalent to the one given by Gauss in \cite{disquisitiones}\footnote{Recall that Gauss defined a form to be reduced if $|\sqrt{\Delta} - 2 |A|| <  B<\sqrt{\Delta}$.}. The equivalence of the two definitions is folklore. 

The $\psl$ class of an indefinite binary quadratic form contains more than one reduced form as opposed to definite binary quadratic forms where the reduced representative is unique, see \cite[Section 6.8]{bqf/vollmer} or \cite[Section 5.6]{computational/nt/cohen} for further discussion. The classical reduction is the process of acting on a non-reduced form $f= (A,B,C)$ by the matrix
\begin{eqnarray*}
	\rho (f) = \mat{0 & 1}{1 & t (f)} = S  (RS)^{t (f)};
\end{eqnarray*}
where $$ t (f) = \left\{
\begin{array}{crcr}
	\sgn (c) \left\lfloor \frac{b}{2|c|}\right\rfloor							& \hbox{if} & |c| \geq \sqrt{\Delta}\\
	\sgn (c) \left\lfloor \frac{\sqrt{\Delta} + b}{2|c|}\right\rfloor		& \hbox{if} & |c| < \sqrt{\Delta}
\end{array}\right\},$$


\noindent and checking whether the resulting form is reduced or not. It is known that after finitely many steps one arrives at a reduced form, call $f_{o}$. Applying $\rho(f_{o})$ to $f_{o}$ produces again a reduced form. Moreover, after finitely many iterations one gets back $f_{o}$. And this set of reduced indefinite binary quadratic forms is called the cycle of the class.

Our aim is now to reveal the reduction method due to Gauss in terms of \c{c}arks. Recall that every edge of a \c{c}ark may be labeled with a unique coset of the corresponding subgroup. That is to say binary quadratic forms may be used to label the edges of the \c{c}ark by Corollary~\ref{cor:1/to/1/corr/between/directed/carks}. 

Given a hyperbolic element $W$ as a word in $R$, $R^{2}$ and $S$ we define the length of $W$, $\ell (W)$, to be the total number of appearances of $R$, $R^{2}$ and $S$. For instance for $W = RS R^{2}S  (RS)^{2}$, $\ell (W) = 8$.

\begin{lemma}
	Given an indefinite binary quadratic form  (reduced or non-reduced), $f$, let $W_{f}$ be a primitive hyperbolic element corresponding to $f$. Then
	\begin{eqnarray}
		\ell (W_{\rho (f)\ccdot f}) \leq \ell (W_{f}).
	\end{eqnarray}
	\label{lemma:length/of/reduced/form/is/minimal}
\end{lemma}
Let us assume from now on that our \c{c}arks are embedded into an annulus, with an orientation which we will assume to be the usual one.\footnote{Although theoretically unnecessary, the choice of an orientation will simplify certain issues. For instance, we shall see that inversion in the class group is reflection with respect to spine.} In addition we also introduce the following shorter notation for our \c{c}arks: in traversing the spine (in either direction) if there are $n$ consecutive Farey branches in the direction of the same boundary component, then we denote this as a single Farey component and write $n$ on the top of the corresponding branch, see Figure~\ref{fig:cark/and/its/short/form}. We will call such \c{c}arks \emph{weighted}. 


\begin{definition}
	Let $\ch$ be a weighted \c{c}ark. Edges of the spine are called \emph{semi-reduced}. In particular, an edge on the spine of $\ch$ is called \emph{reduced} if and only if it is on the either side of a Farey component which is in the direction of the inner boundary component.
	\label{defn:reduced/edge}
\end{definition}

\begin{figure}[h!]
	\begin{subfigure}{5cm}
		\centering
		\includegraphics[scale=0.35]{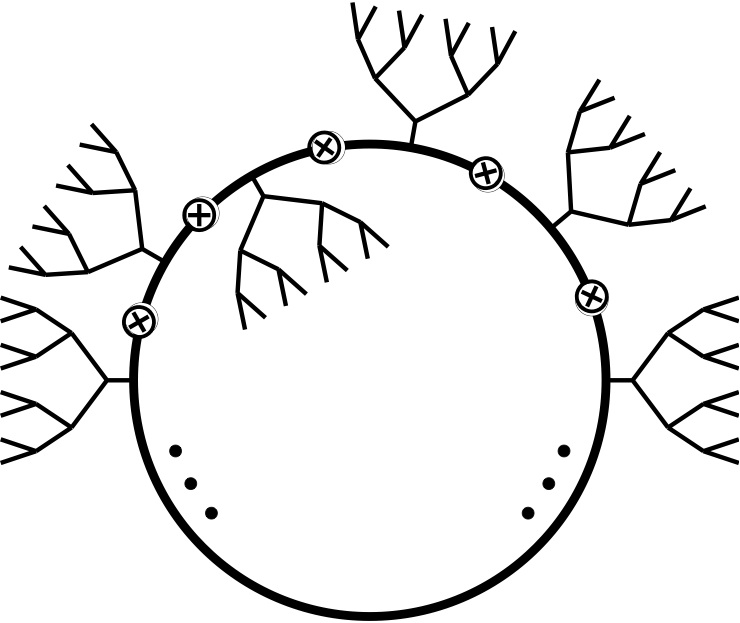}
	\end{subfigure}
	\begin{subfigure}{5cm}
		\centering
		\includegraphics[scale=0.39]{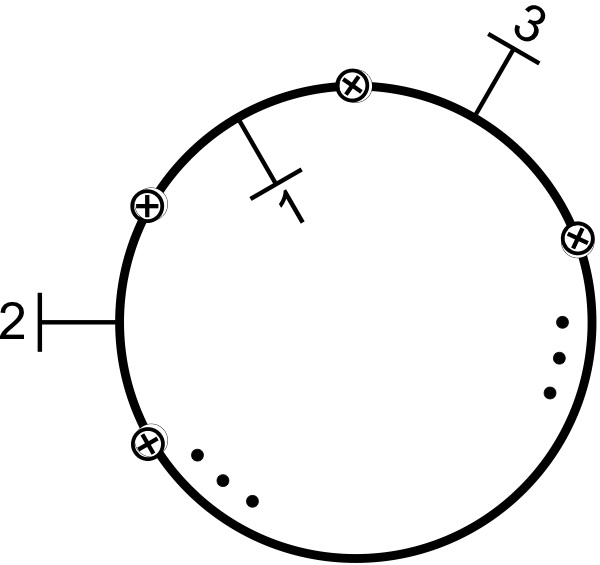}
	\end{subfigure}
	\caption{A \c{c}ark and its short form. }
\end{figure}

\noindent Remark that as we have fixed our orientation to be the usual one, there is no ambiguity in this definition. In addition note that semi-reduced edges are in one to one correspondence  between the forms $f =  (A,B,C)$ in a given class for which $AC<0$. We are now ready to describe reduction theory of binary quadratic forms in terms of \c{c}arks. We have seen that multiplication by the matrix $\rho (f)$ is, in general, the process of moving the base edge of the \c{c}ark to the spine as a result of Lemma~\ref{lemma:length/of/reduced/form/is/minimal}. However, this is not enough. That is, not every edge on the spine corresponds to a reduced form. Reduced forms correspond to edges where the Farey branches switch from one boundary component to the other. More precisely, we have:

\begin{theorem}
	Reduced forms in an arbitrary indefinite binary quadratic form class $[f]$ are in one to one correspondence between the reduced edges of the \c{c}ark corresponding to the given class.
	\label{thm:reduced/edges/reduced/bqfs}
\end{theorem}

\noindent As we have remarked the action of $\psl$ on binary quadratic forms is equivalent to the change of base edge on the set of \c{c}arks. Hence the above Theorem is an immediate consequence of the following:

\begin{lemma}
	Let $\c{C}_{f}$ denote the \c{c}ark associated to an arbitrary indefinite binary quadratic form $f$. The reduction operator $\rho (f)$ is transitive on the set of reduced edges of $\c{C}_{f}$.
\end{lemma}

Let us give some examples:

\begin{example}
	Let us consider the form $f =  (7,33,-15)$. It is easy to check that $f$ is reduced. $W_{f} =  (R^{2}S)^2 \,  (RS)^2 \,  R^2 S \, RS \,  (R^{2}S)^{7} \,  (RS)^{5} = \mat{-38 & -195}{ -91 & -467}$. The trace of the class is $-505$. By Gauss' theory the class $[f]$ is an element in the quadratic number field with discriminant $1509$.
\end{example}

\begin{figure}
	\centering
	\includegraphics[scale=0.23]{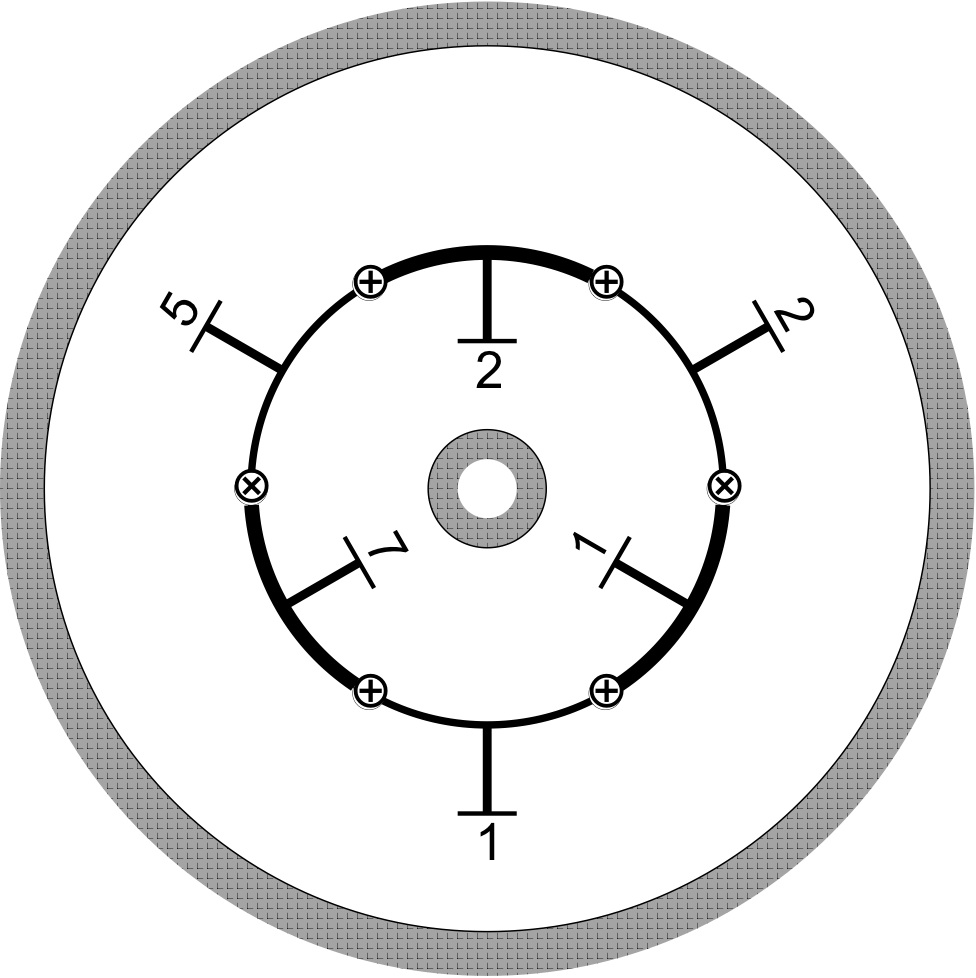}
	\caption{ \c{C}ark corresponding to the class represented by the form $ (7,33,-15)$. Bold edges are reduced.}
	\label{fig:cark/general/example}
\end{figure}

\begin{example}
	Let $\Delta = n^{2}+4n$ for some positive integer $n$. Then the identity in the class group is given by the \c{c}ark in Figure~16a and the corresponding form is $ (-n,n,1)$. If $\Delta = n^{2}+4$, then the identity is represented by the form $\frac{1}{n} (-n,n^{2},n) =  (1,n,-1)$. The corresponding \c{c}ark has two Farey branches, see Figure~16b.
	\label{ex:identities/having/two/Farey/components}
\end{example}

\begin{figure}[h!]
	\begin{subfigure}[]{4.5cm}
		\centering
		\includegraphics[scale=0.23]{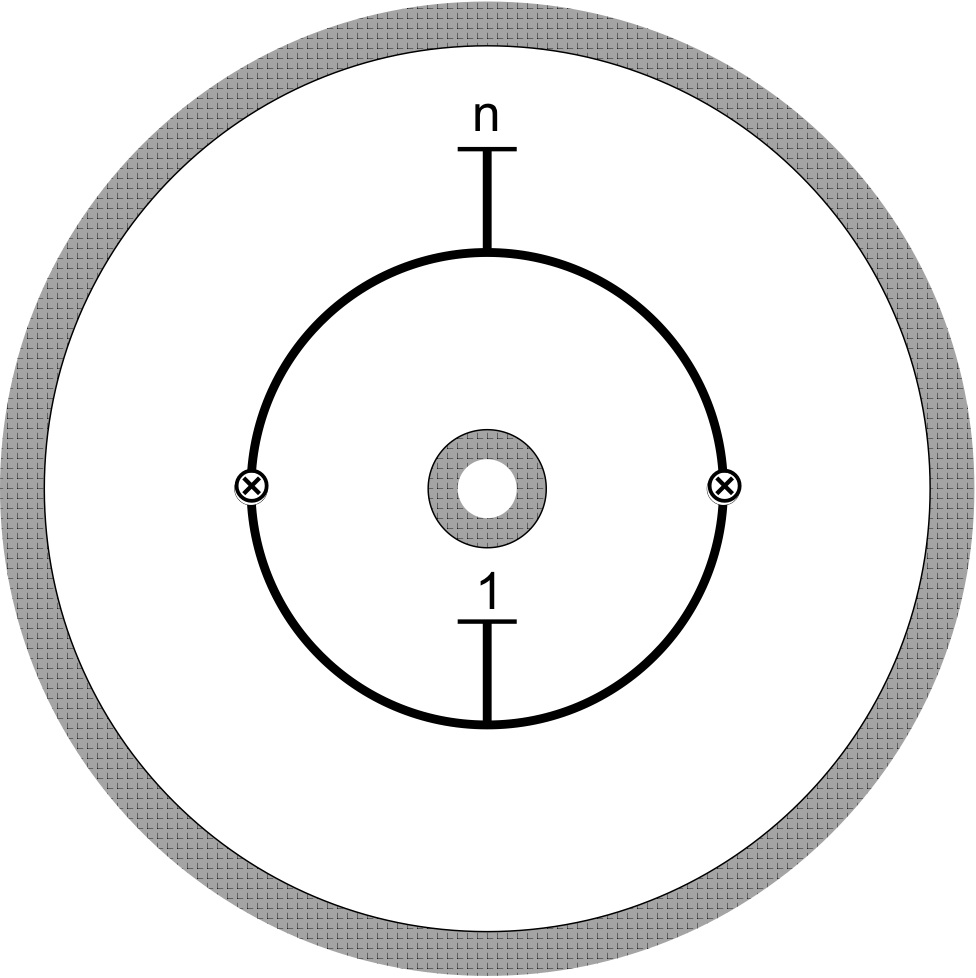}
		\caption{\scriptsize (a) Identity for $\Delta = n^{2} + 4n$.}
		\label{fig:identity/n2}
	\end{subfigure}
	\qquad
	\begin{subfigure}[]{4cm}
		\centering
		\includegraphics[scale=0.23]{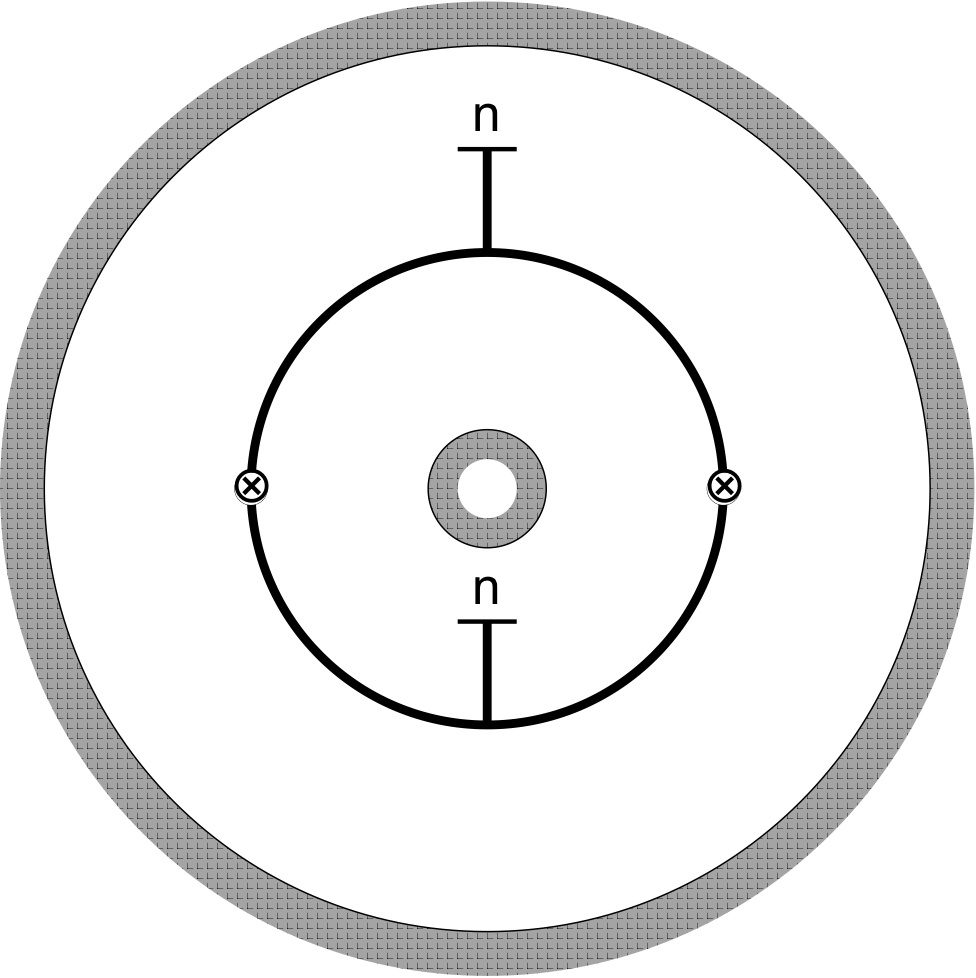}
		\caption{\scriptsize (b) Identity for $\Delta = n^{2}+4$.}
		\label{fig:identity/5n2}
	\end{subfigure}
	\caption{ \ }
	\label{fig:identities}
\end{figure}

\noindent However, one has to admit that there are very complicated \c{c}arks representing the identity of the class group. For instance, the \c{c}ark corresponding to the form $ (-7,23,16)$ has 42 Farey branches.

\subsection{Ambiguous and Reciprocal forms}

Let us now discuss certain symmetries of a \c{c}ark. For a given \c{c}ark $\ch$ let $\ch^{r}$ be the \c{c}ark which is the mirror image of $\ch$ about any line passing through the `center" of the spine (assuming that the Farey components coming out of the spine in its shorter notation that we have introduced is evenly spaced). It is easy to see that both ideal classes represented by the two \c{c}arks $\ch$ and $\ch^{r}$ have the same discriminant. A straightforward computation leads to the following:

\begin{proposition}
	Given a \c{c}ark $\ch$ the binary quadratic form class represented by $\ch^{r}$ is inverse of the class represented by $\c{C}$.
	\label{propn:inverse/carks}
\end{proposition}

\begin{example}
	Let us consider the form $f =  (-2377, 10173, 1349)$ having discriminant $116316221$. The form $g =  (-4027, 8915, 2287)$ is an element in the ideal class represented by this form. The corresponding \c{c}arks are shown in Figure~\ref{fig:inverse/cark}. The forms are inverses of each other.
	\label{ex:inverse}
\end{example}

\begin{figure}[h!]
	\centering
	\begin{subfigure}{4cm}
		\centering
		\includegraphics[scale=0.23]{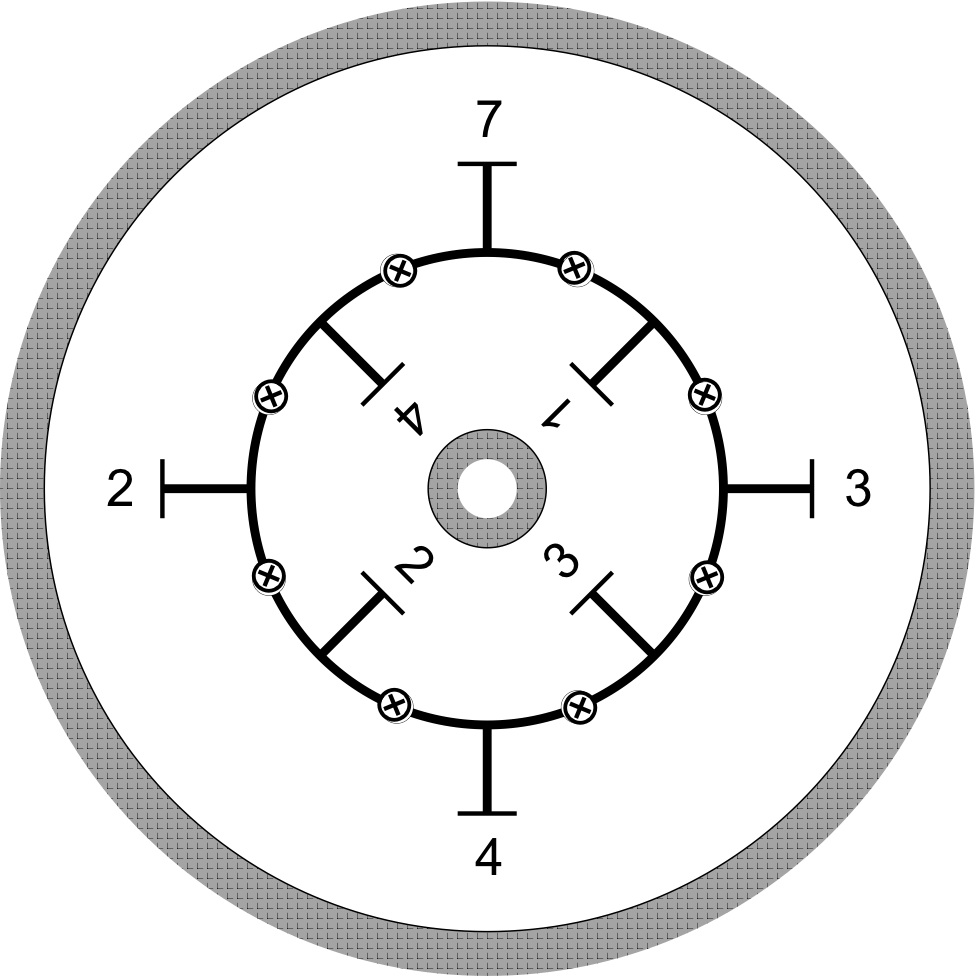}
		\caption{\scriptsize \c{C}ark corresponding to $f =  (-2377, 10173, 1349)$.}
	\end{subfigure}
	\begin{subfigure}{4cm}
		\centering
		\includegraphics[scale=0.23]{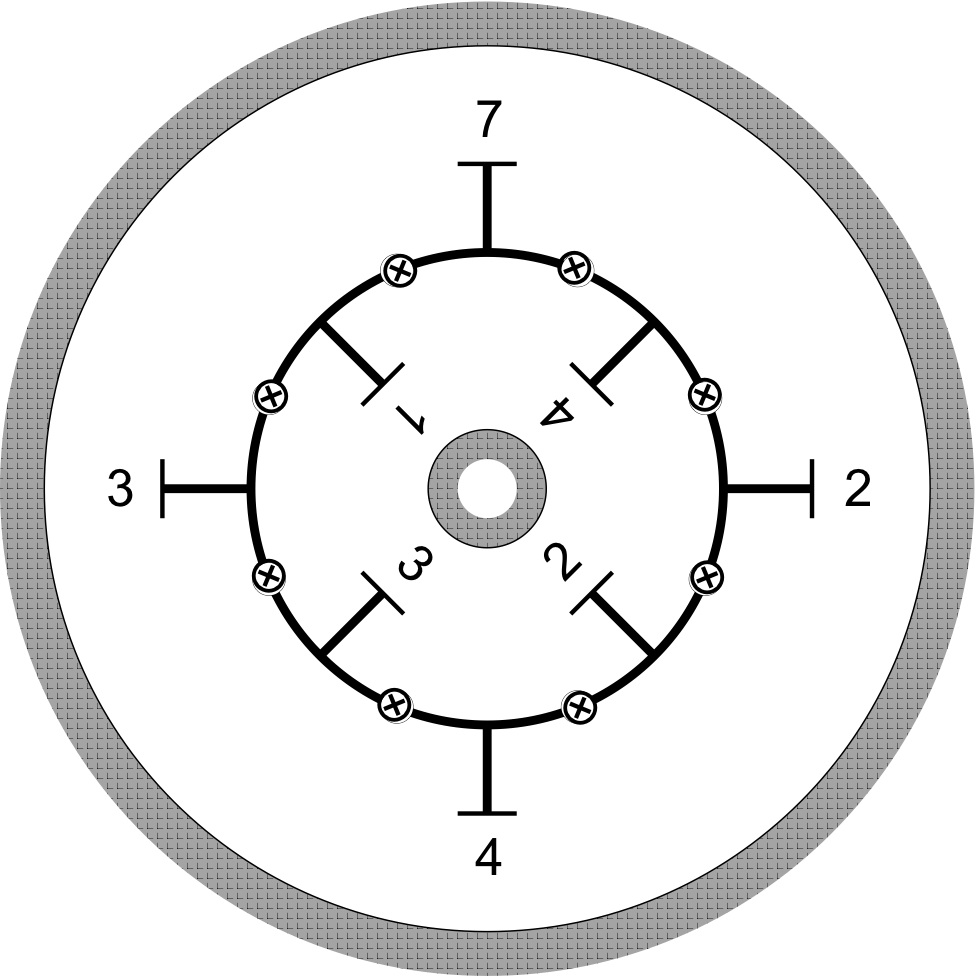}
		\caption{\scriptsize \c{C}ark corresponding to $f^{-1} =  (-4027, 8915, 2287)$.}
	\end{subfigure}
	\begin{subfigure}{4cm}
		\centering
		\includegraphics[scale=0.23]{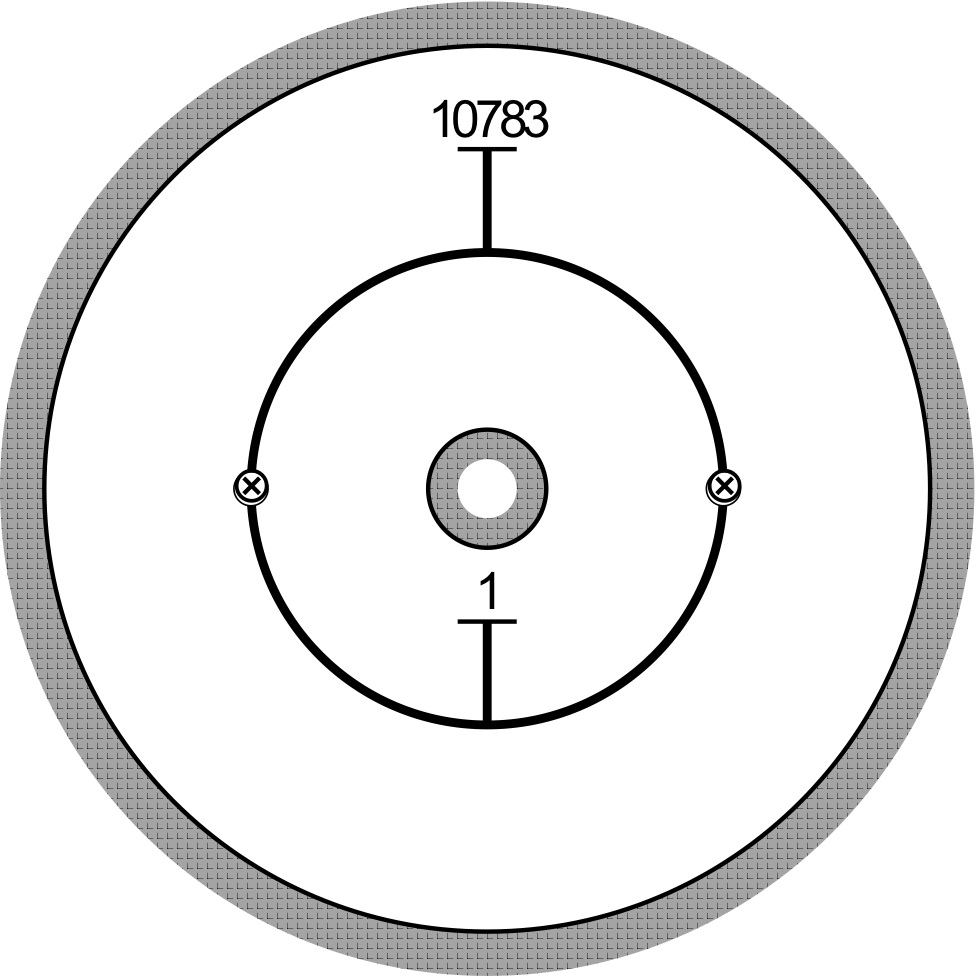}
		\caption{\scriptsize \c{C}ark of the product of $f \times f^{-1}$.}
	\end{subfigure}

	\caption{ Two \c{c}arks inverses of one another and their product.}
	\label{fig:inverse/cark}
\end{figure}

Recall that Gauss has defined a binary quadratic form to be ambiguous if it is equivalent to its inverse or equivalently if the corresponding equivalence class contains $ (a,ka,c)$ for some $a$, $c$ and $k$. Following Gauss, we define a \c{c}ark $\ch$ \emph{ambiguous} if $\ch$ and $\ch^{r}$ are isomorphic as \c{c}arks, or equivalently correspond to the same subgroup of $\psl$. So from Proposition~\ref{propn:inverse/carks} we deduce:

\begin{corollary}
	Ambiguous \c{c}arks correspond to ambiguous forms.
	\label{cor:ambiguous/carks}
\end{corollary}

\noindent In addition to all the examples considered in Example~\ref{ex:identities/having/two/Farey/components}, which represent ambiguous classes as they are of the form $ (a, ka, c)$, let us give one more example:

\begin{example}
	Consider the form $f =  (3,18,-11)$. The form is reduced and ambiguous as one immediately checks. The corresponding \c{c}ark is given in Figure~\ref{fig:ambiguous}
	\label{ex:ambiguous/carks}
\end{example}

\begin{figure}[h!]
	\centering
	\includegraphics[scale=0.23]{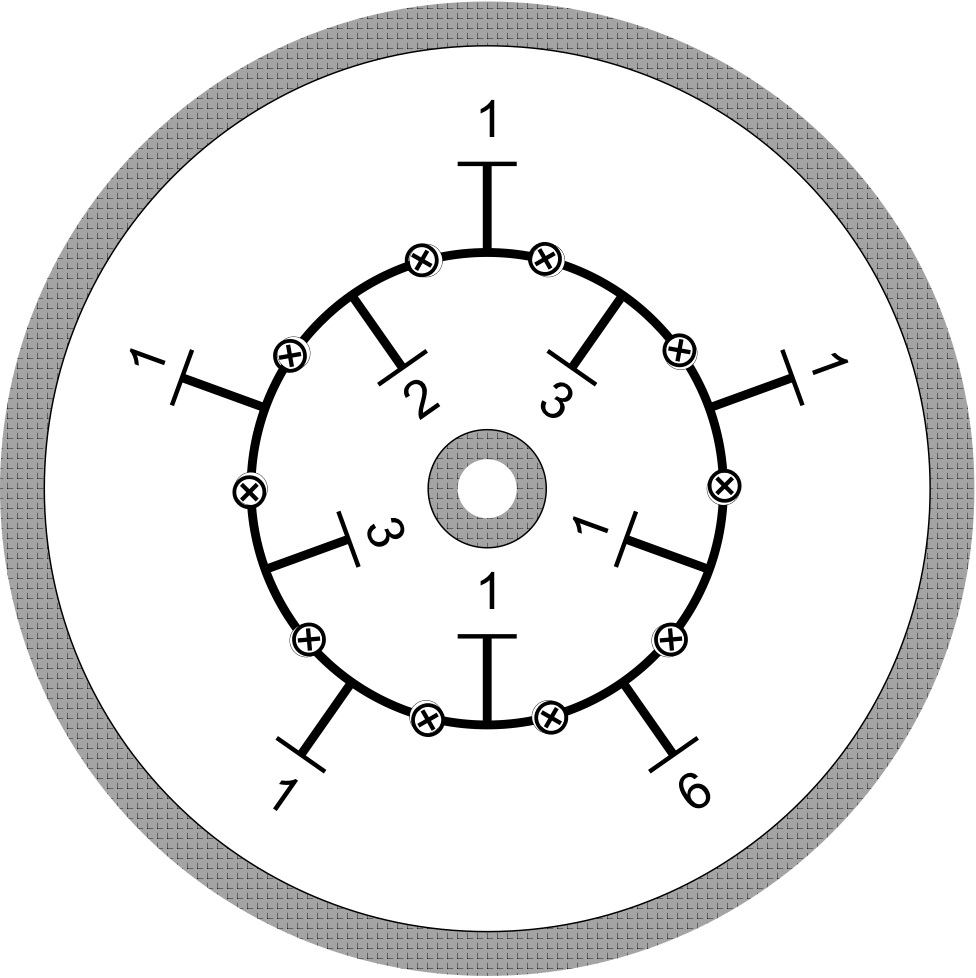}
	\caption{\c{C}ark corresponding to the ambiguous form $f =  (3,18,-11)$.}
	\label{fig:ambiguous}
\end{figure}

Let us now discuss ``rotational" symmetries. In Section~\ref{sec:necklace/bracelet} we defined a directed \c{c}ark with a base edge \emph{primitive} if and only if its spine is not periodic. Let $\mathfrak{c}_{prim}$ denote the set of primitive \c{c}arks. It is easy to see that primitive hyperbolic elements\footnote{Recall that an element $M\in \psl$ is said to be {\it primitive} if it is not a positive power of another element of the modular group.} in $\psl$ correspond to primitive \c{c}arks or equivalently to prime geodesics in $\mathbb{H}$. 

\begin{corollary}
	There is a one to one correspondence between the following two sets:
	\begin{center}
		\begin{tabular}{c c c}
			$\mathfrak{c}_{prim}$ & $\longleftrightarrow$ &$ \bigg \{$\begin{tabular}{c} $\psl$ classes of primitive \\ indefinite binary quadratic forms \\ having discriminant $\Delta \in \mathcal{D}$ \end{tabular} $\bigg\}$
		\end{tabular}
	\end{center}
\end{corollary}

Finally, let $\ch^{m}$ denote the mirror of a given \c{c}ark, that is the \c{c}ark obtained by reflecting $\ch$ with respect to the spine. Once again observe that both $\ch$ and $\ch^{m}$ have the same discriminant. In fact, an indefinite binary quadratic form say $f =  (A,B,C)$ is given which is represented by the \c{c}ark $\ch$ then the \c{c}ark $\ch^{m}$ represents the form $f' =  (-A,B,-C)$ and the same holds for every element in $[f]$. We conclude that both \c{c}arks represent ideal classes that have the same order in the class group.

Let $W$ be a hyperbolic element in $\psl$. In \cite{sarnak/reciprocal/geodesics}, Sarnak has defined $W$ to be reciprocal if $W$ is conjugate to its inverse. The conjugation turns out to be done by a unique element  (up to multiplication by an element in $\langle W \rangle$) of order $2$, and thus reciprocal elements correspond to dihedral subgroups of the modular group\footnote{Remember that primitive \c{c}arks correspond to maximal $\Z$-subgroups of $\psl$.}. A form $f =  (A,B,C)$ is called reciprocal if  $C = -A$. It is known that reciprocal hyperbolic elements correspond to reciprocal indefinite binary quadratic forms, \cite{sarnak/reciprocal/geodesics}. In a similar fashion we call a \c{c}ark \emph{reciprocal} if $\ch$ and $ (\ch^{m})^{r}$ are isomorphic as \c{c}arks. In fact since two operators $\cdot^{m}$ and $\cdot^{r}$ commute, if $\ch$ is a reciprocal \c{c}ark then so is $\ch^{m}$.

\begin{proposition}
	Reciprocal forms correspond to reciprocal \c{c}arks.
	\label{propn:reciprocal/carks}
\end{proposition}

\begin{figure}[H]
	\centering
	 \includegraphics[scale=0.75]{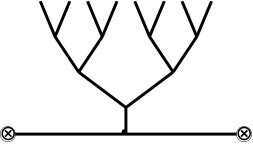}
	\caption{The graph $\mathcal{F}/\langle  S,R^2SR  \rangle $}
	\label{fig:quotient/by/hyperbolic}
\end{figure}

\begin{figure}[H]
	\centering
	 \includegraphics[scale=0.75]{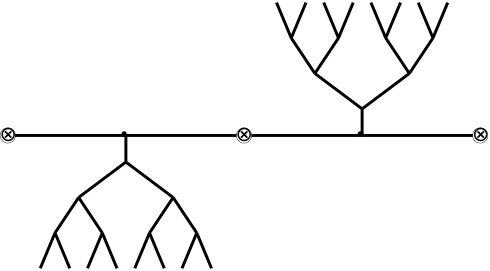}
	\caption{ The graph $\mathcal{F}/\langle  RSR^2,S (RSR^2)S  \rangle $}
	\label{fig:quotient/by/hyperbolic}
\end{figure}

\begin{example}
	Consider the form $f =  (-8,11,8)$. The corresponding hyperbolic element in $\psl$ is $\mat{101 & -192}{-192 & 365}$. The corresponding \c{c}ark is shown in Figure~\ref{fig:reciprocal}, where it is easy to see that $\ch$ and $ (\ch^{m})^{r}$ are same.
	\label{ex:reciprocal/cark}
\end{example}

\begin{example}[Reciprocal Identities]
	The forms $f =  (1,n^{2},-1)$ already appeared in Example~\ref{ex:identities/having/two/Farey/components} are reciprocal and represent identity in the class group. Note also that such forms come from the word $ (R^{2}S)^{n} (RS)^{n}$. The \c{c}arks of these reciprocal identities are in Figure~\ref{fig:identity/5n2}.
	\label{ex:reciprocal/identities}
\end{example}

\begin{figure}[h!]
	\centering
	\includegraphics[scale=0.23]{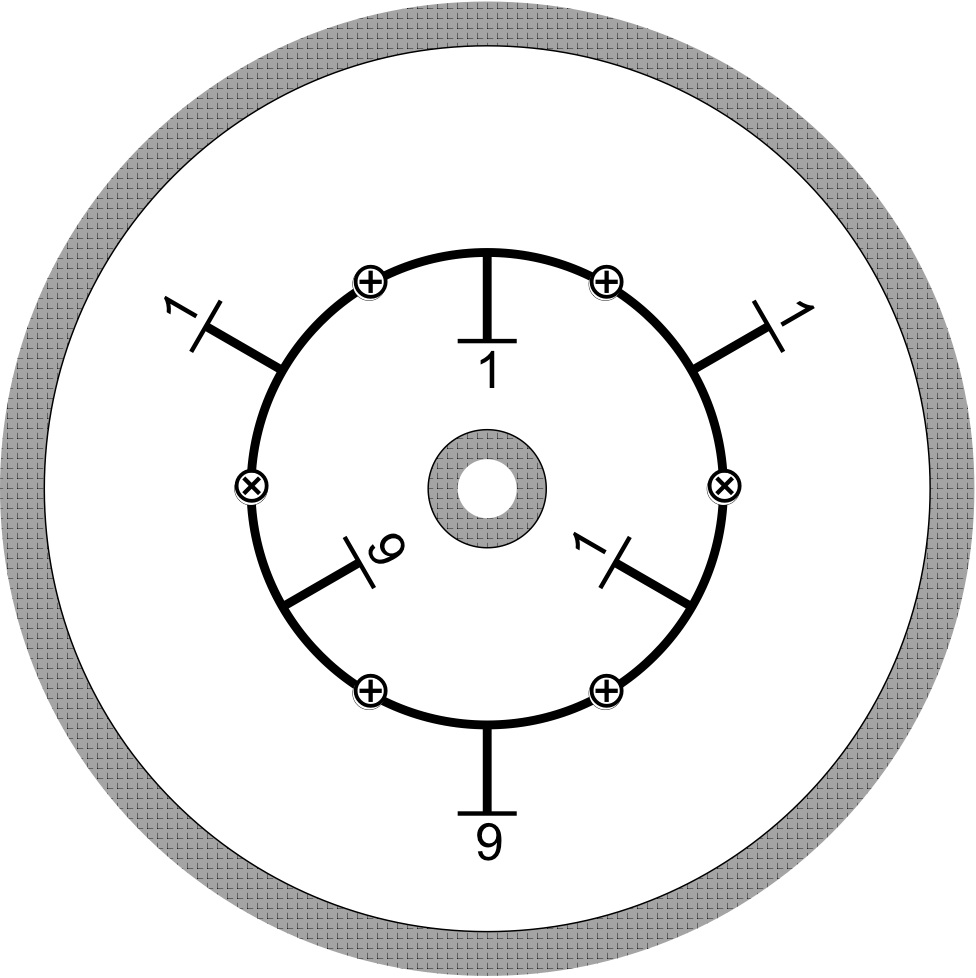}
	\caption{\c{C}ark corresponding to the reciprocal form $f =  (-8,11,8)$.}
	\label{fig:reciprocal}
\end{figure}

\subsection{Miscellany}
Binary quadratic forms is a central and classical topic and have connections to diverse fields. Here we touch upon some of these.

\subsubsection{Computational Problems}
There are several important computational problems related to \c carks, in connection with the class number problems in the indefinite case. The most basic invariant of a \c cark is the length of its spine. The  (absolute) trace of the associated matrix is another, much subtler invariant. The problem of listing \c{c}arks of the same trace is equivalent to the problem of computing class numbers. Also, the Gauss product on classes of forms defines an abelian group structure on the set \c{c}arks of the same trace, namely the {\it class group}. It is a work in progress to reach to a new understanding of class groups in terms of the graphical representation of their elements by  \c carks.


\subsubsection{Closed geodesics on the modular surface.}
Let us note in passing that primitive \c{c}arks parametrize closed geodesics on the modular curve, and so \c{c}arks are closely connected to symbolic dynamics on the modular curve, 
see \cite{katok/ugarcovici}, encoding of geodesics, and Selberg's trace formula, see \cite{zagier/new/points/of/view/on/the/selberg/zeta}.

\subsubsection{The Markoff number of an indefinite binary quadratic form.}
There is an arithmetic invariant of indefinite binary quadratic forms called the Markoff value $\mu (F)$ which is defined as
$$
\mu (F):= \frac{\sqrt{\Delta (F)}}{m(f)} , \mbox{ where } 
m(f):={\min_{ (x,y)\in \Z^2 \setminus \{ (0,0)\}} |F (x,y)|} 
$$

Alternatively one can run over the class of $F$ and compute the minima of equivalent forms at a fixed point $p_0$, for example $ (x,y)= (0,1)$. 
Hence the choice of this fixed point $p_0$ defines a function on the set of edges of the associated \c{c}ark, and the Markoff value of the form is the maximal value attained by this function defined on the \c{c}ark. There are also \c{c}arks associated to Markoff irrationalities which we call {\it Markoff \c{c}arks}. A solution to the representation problem of indefinite binary quadratic forms is given in \cite{reduction} and as a by-product Markoff value of a given form can be computed. The algorithms will be available within the software developed by the first two authors and their collaborators, \cite{sunburst}.

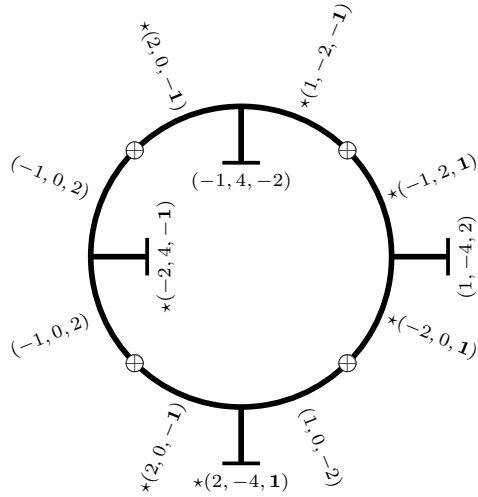
\begin{figure}[H]
	\centering
	\begin{tikzpicture} [scale=0.5]

		\draw[line width=0.75mm]  (0,0) circle  (4cm);

		\draw[line width=0.75mm]  (0,-4)-- (0,-5.5);
		\draw[line width=0.75mm]  (0,4)-- (0,2.5);
		\draw[line width=0.75mm]   (-4,0)-- (-2.5,0);
		\draw[line width=0.75mm]  (4,0)-- (5.5,0);

		\draw[line width=0.5mm,rotate=90]  (-0.5,2.5)-- (0.5,2.5);
		\draw[line width=0.5mm,rotate=90]  (-0.5,-5.5)-- (0.5,-5.5);
		\draw[line width=0.5mm]  (-0.5,2.5)-- (0.5,2.5);
		\draw[line width=0.5mm]  (-0.5,-5.5)-- (0.5,-5.5);

		\node at  (0,2) [ultra thick]{\scriptsize $ (-1,4,-2)$};
		\node at  (0,-6) [ultra thick]{\scriptsize $\star (2,-4,{\mathbf 1})$};
		\node at  (6,0) [ultra thick,rotate=90]{\scriptsize $ (1,-4,2)$};
		\node at  (-2,0) [ultra thick,rotate=90]{\scriptsize $\star (-2,4,-{\mathbf 1})$};
		\node at  (5.5*cos{22.5},5.5*sin{22.5}) [ultra thick,rotate=22.5]{\scriptsize $\star (-1,2,{\mathbf 1})$};
		\node at  (5.5*cos{67.5},5.5*sin{67.5}) [ultra thick,rotate=67.5]{\scriptsize $\quad \star (1,-2,-{\mathbf 1})$};
		\node at  (-5.5*cos{22.5},5.5*sin{22.5}) [ultra thick,rotate=-22.5]{\scriptsize $ (-1,0,2)$};
		\node at  (-5.5*cos{67.5},5.5*sin{67.5}) [ultra thick,rotate=-67.5]{\scriptsize $\star (2,0,-{\mathbf 1})$};
		\node at  (5.5*cos{22.5},-5.5*sin{22.5}) [ultra thick,rotate=-22.5]{\scriptsize $\star (-2,0,{\mathbf 1})$};
		\node at  (5.5*cos{67.5},-5.5*sin{67.5}) [ultra thick,rotate=-67.5]{\scriptsize $ (1,0,-2)$};
		\node at  (-5.5*cos{22.5},-5.5*sin{22.5}) [ultra thick,rotate=22.5]{\scriptsize $ (-1,0,2)$};
		\node at  (-5.5*cos{67.5},-5.5*sin{67.5}) [ultra thick,rotate=67.5]{\scriptsize $\star (2,0,-{\mathbf 1})$};


		\draw  (4*cos{45},4*sin{45}) circle  (2.25mm);
		\fill[white]  (4*cos{45},4*sin{45}) circle  (2.25mm);
		\node at  (4*cos{45},4*sin{45}) [ultra thick]{ $\oplus$};
		\draw  (4*cos{-45},4*sin{-45}) circle  (2.25mm);
		\fill[white]  (4*cos{-45},4*sin{-45}) circle  (2.25mm);
		\node at  (4*cos{-45},4*sin{-45}) [ultra thick]{ $\oplus$};
		\draw  (4*cos{135},4*sin{135}) circle  (2.25mm);
		\fill[white]  (4*cos{135},4*sin{135}) circle  (2.25mm);
		\node at  (4*cos{135},4*sin{135}) [ultra thick]{ $\oplus$};
		\draw  (4*cos{-135},4*sin{-135}) circle  (2.25mm);
		\fill[white]  (4*cos{-135},4*sin{-135}) circle  (2.25mm);
		\node at  (4*cos{-135},4*sin{-135}) [ultra thick]{ $\oplus$};
	\end{tikzpicture}
	\caption{Minimum edges of $\F / Aut(\{ (1,0,-2)\})$($\star$ stands for forms which attain the minimum.).}
	\label{fig:quotient/by/hyperbolic}
\end{figure}

To conclude the paper, let us rephrase our main result: we show that the class of every primitive indefinite binary quadratic form is not simply a set but it has the extra structure of an infinite graph, namely a \c{c}ark, such that the forms in the class are identified with the edges of the graph. This graph admits a topological realization as a subset of an annulus and explains very well some known phenomena around Gauss' reduction theory of forms and Zagier's reduction of elements of $\psl$ as explained in \cite{katok/ugarcovici}. In our point of view both Gauss reduced forms and Zagier reduced forms correspond to edges on the what we call spine of the \c{c}ark. Various properties of forms and their classes are manifested in a natural way on the \c{c}ark. The first instance of such a question concerning binary quadratic forms has been addressed by the second named author in \cite{reduction}, where he has given an improvement of Gauss' reduction of binary quadratic forms, and has given solutions to the minimum problem and representation problem of binary quadratic forms.

\paragraph{Acknowledgements.}
\label{ackref}
The first named author is thankful to Max Planck Institute at Bonn for their hospitality during the preparation of the current paper. This research has been funded by the T\"UB\.ITAK grant 110T690. The first named author was also funded by the Galatasaray University Research Grant 12.504.001. The second named author is funded by Galatasaray University Research Grant 13.504.001.


%
\end{document}